\documentclass[reqno]{amsart}
\usepackage[foot]{amsaddr}
\usepackage{amssymb, amsmath}

\usepackage{mathpazo,courier}
\usepackage[scaled]{helvet}

\usepackage[usenames,dvipsnames]{xcolor}
\usepackage[pagebackref,colorlinks,linkcolor=BrickRed,citecolor=OliveGreen,urlcolor=Blue,hypertexnames=true]{hyperref}
\usepackage{enumitem}

\usepackage{cleveref}
\crefname{equation}{}{}

\newcommand\N{\mathbb N}

\newcommand\Q{\mathbb Q}
\newcommand\R{\mathbb R}

\newcommand\COP{\mathrm{COP}}
\newcommand\SPN{\mathrm{SPN}}

\newcommand\MK{\mathcal{K}}
\newcommand\MQ{\mathcal{Q}}

\newcommand\MS{\mathcal{S}}

\newcommand\poly{\mathrm{poly}}
\newcommand\size{\mathrm{size}}

\newcommand{\ignore}[1]{}

\newcommand{\ep}{\varepsilon}

\theoremstyle{definition}
\newtheorem{thm}{Theorem}
\numberwithin{thm}{section}
\newtheorem{lemma}[thm]{Lemma}
\newtheorem{cor}[thm]{Corollary}

\newtheorem{obs}[thm]{Observation}
\newtheorem{ex}[thm]{Example}

\newtheorem{df}[thm]{Definition}

\newtheorem*{rep@theorem}{\rep@title}
\newcommand{\newreptheorem}[2]{%
\newenvironment{rep#1}[1]{%
 \def\rep@title{#2 \ref{##1}}%
 \begin{rep@theorem}}%
 {\end{rep@theorem}}}

\newreptheorem{pro}{Proposition}
\newreptheorem{thm}{Theorem}

\title{Computational complexity of sum-of-squares bounds for copositive programs}

\author{Marilena Palomba*}
\address{*SUPSI, IDSIA, Via la Santa 1, 6962 Lugano-Viganello, Switzerland.}
\email{marilena.palomba@supsi.ch}

\author{Lucas Slot$^\dagger$}
\address{$^\dagger$ETH Zurich. Andreasstrasse 5, 8092 Zurich, Switzerland.}
\email{lucas.slot@inf.ethz.ch}

\author{Luis Felipe Vargas*}
\email{luis.vargas@supsi.ch}

\author{Monaldo Mastrolilli*}
\email{monaldo.mastrolilli@supsi.ch}


\date{\today}


\begin{document}
\begin{abstract}
In recent years, copositive programming has received significant attention for its ability to model hard problems in both discrete and continuous optimization. Several relaxations of copositive programs based on semidefinite programming (SDP) have been proposed in the literature, meant to provide tractable bounds. However, while these SDP-based relaxations are amenable to the ellipsoid algorithm and interior point methods, it is not immediately obvious that they can be solved in polynomial time (even approximately). In this paper, we consider the sum-of-squares (SOS) hierarchies of relaxations for copositive programs introduced by Parrilo~(2000), de Klerk \& Pasechnik~(2002) and Peña, Vera \& Zuluaga~(2006), which can be formulated as SDPs. We establish sufficient conditions that guarantee the polynomial-time computability (up to fixed precision) of these relaxations. These conditions are satisfied by copositive programs that represent standard quadratic programs and their reciprocals. As an application, we show that the SOS bounds for the (weighted) stability number of a graph can be computed efficiently. 
Additionally, we provide pathological examples of copositive programs (that do not satisfy the sufficient conditions) whose SOS relaxations admit only feasible solutions of doubly-exponential size.
\end{abstract}

\maketitle

\section{Introduction}

A symmetric matrix $M\in \mathcal S^n$ is \emph{copositive} if $x^T M x\geq 0$ for all $x\in\mathbb R^n_+$, i.e., if the quadratic form $x^T M x = \sum_{i,j=1}^n M_{ij} x_i x_j$ associated to $M$ is nonnegative for all vectors $x \in \mathbb R^n_+$. The set of all copositive matrices is a full-dimensional, closed and convex matrix cone~\cite{Hall_Newman_1963}, called the {\em copositive cone}:
\begin{equation}\label{COP1}
    \COP_n := \{M\in \mathcal S^n \colon x^T M x\geq 0 \ \forall x\in \mathbb R^n_+\}.
\end{equation}
Copositivity of $M$ can equivalently be defined using the \emph{quartic} form $(x^{\circ 2})^T M x^{\circ 2}$, where $x^{\circ 2} = (x_1^2,\dots, x_n^2)$. Namely, we have
\begin{equation}\label{COP2}
    \COP_n = \{M \in \mathcal S^n : (x^{\circ 2})^T M x^{\circ 2} \geq 0\ \forall x \in \mathbb R^n\}.
\end{equation}
The copositive cone contains both the cone $\mathcal{S}^+_n$ of positive semidefinite matrices, as well as the cone $\mathcal N_n$ of (symmetric) matrices with nonnegative entries. That is, ${\COP_n \supseteq \SPN_n := \mathcal{S}^+_n +  \mathcal N_n}$. This inclusion is strict if and only if $n \geq 5$ \cite{Diananda_1962, Hall_Newman_1963}.

Let $A_1, \dots, A_m, C\in \MS^n$ be symmetric matrices and let $y\in \mathbb{R}^m$ be a vector. A \textit{copositive program} is defined as follows  
\begin{align}\label{cop-program}\tag{CP}
    p^* = \min\Big\{ b^Ty: \sum_{i=1}^my_iA_i-C \in \COP_n\Big\}.
\end{align}
Copositive programming has practical implications in several fields, including operations research (for optimizing resource allocation and scheduling problems, see e.g. \cite{Kumar1996,Humes1997}), engineering (for dynamical systems and optimal control, see e.g. \cite{BUNDFUSS2009342,Mesbahi2000}) and game theory (for refinements of the Nash equilibrium concept, see e.g. \cite{Bomze1989GameTF}). See \cite{BomzeSchachinger2012} for additional references and applications.

On the other hand, copositive programming is NP-hard in general. Indeed, classical NP-hard problems like computing the stability and chromatic number of a graph can be formulated as copositive programs. 

\begin{ex}\label{example:alpha}
    The \emph{stability number} $\alpha(G)$ of a graph $G$ measures the cardinality of the largest set of vertices in $G$ such that no two vertices in the set are adjacent.
    As shown by Motzkin \& Straus~\cite{Motzkin_Straus_1965}, the stability number can be computed via a quadratic optimization problem over the standard simplex:
    \begin{equation} \label{EQ:alpha_SQP}
        \frac{1}{\alpha(G)} = \min \{x^T(A_G+I)x \colon x\in\Delta_n\},
    \end{equation}
    where $A_G$ is the adjacency matrix of $G$.
    De Klerk \& Pasechnik \cite{dKP02} showed that~$\alpha(G)$ can be computed by solving the following copositive program:
    \begin{equation}\label{alpha}
        \alpha(G) =\min \{t\colon t(A_G+I)-J \in \COP_n\},
    \end{equation}
    where $J$ is the all-ones matrix. 
\end{ex}
So, while copositive programs are widely applicable, it is generally intractable to solve them to optimality. This has motivated the definition of several hierarchies of (tractable) relaxations of~\eqref{cop-program} in the literature. In this paper, we consider relaxations that rely on modeling nonnegativity of the quadratic (or quartic) form associated to a symmetric matrix using \emph{sums of squares} of polynomials of increasing degree.
In both cases, this leads to a sequence of \emph{spectrahedral cones} $\mathcal C^{\smash{(1)}} \subseteq \mathcal C^{\smash{(2)}} \subseteq \ldots \subseteq \COP_n$ that approximate the copositive cone.
These cones can be used to define an increasingly accurate sequence of \emph{upper bounds} $p_1 \geq p_2 \geq \ldots \geq p^*$ on the optimum of the program~\eqref{cop-program}:
\begin{align*}
    p^{(k)} := \min\Big\{ b^Ty: \sum_{i=1}^my_iA_i-C \in \mathcal C^{(k)}\Big\} \geq p^*.
\end{align*}
The upshot is that, for any fixed $k \in \N$, the bound $p^{\smash{(k)}}$ can be formulated as a \emph{semidefinite program} involving matrices of polynomial size in $n$. In contrast to copositive programs, semidefinite programs can be tackled by computationally efficient algorithms such as the ellipsoid algorithm and interior point methods. In fact, under some seemingly minor conditions, they can be solved in polynomial time (up to a small additive error), see, e.g.,~\Cref{thm:approx sol} below.

However, and importantly, it is not clear a priori whether these conditions are actually satisfied for the semidefinite formulations of the programs that define the bounds $p^{\smash{(k)}}$. As a result, we cannot conclude that these bounds can be computed in polynomial time (even approximately).
In fact, even if the original program~\eqref{cop-program} is defined in terms of matrices with well-behaved entries (i.e., of small \emph{bit size}, see~\Cref{SEC:bitsize}), it can happen that any feasible solution of the semidefinite relaxation must contain entries of doubly-exponential size (in $n$).
These feasible solutions cannot be identified in polynomial time by standard implementations of the ellipsoid algorithm or interior point methods. 
In Section \ref{section:examples} we provide explicit examples that exhibit this pathology (similar examples were provided by O'Donnell~\cite{odonnell:LIPIcs.ITCS.2017.59} in the context of polynomial optimization, see~\Cref{SEC:related-word} below). 

The goal of this paper is to identify when certain well-studied semidefinite relaxations for copositive programs can provably be computed  in polynomial time (up to any fixed precision). We show that this is possible when~\eqref{cop-program} corresponds to a standard quadratic program, or the reciprocal of a quadratic program (under a mild additional condition).  The latter includes, in particular, program~\eqref{alpha} corresponding to the (weighted) stability number of a graph. Our results give a theoretical justification for the use of sum-of-squares relaxations in copositive programming.
Before we can state our precise contributions, we first introduce the relaxations that we consider in this paper more formally.

\subsection{Sum-of-squares relaxations}\label{SEC:relaxations}

The sum-of-squares (SoS) proof system is a powerful method in optimization and theoretical computer science. 
It is particularly significant in addressing polynomial optimization problems and has connections to several areas, including combinatorial optimization, complexity theory, and algorithm design, see, e.g., \cite{FlemingKothariPitassi19}.

At its core, the SoS proof system deals with proving that a given polynomial is nonnegative, by writing it as a sum of squares of other polynomials.
A polynomial $p \in \mathbb R[x]$ is a \emph{sum of squares} if there exist polynomials $q_1, q_2, \ldots, q_m \in \mathbb R[x]$ such that
\begin{equation}\label{SoS}
    p= \sum_{i=1}^m q_i^2. 
\end{equation}
We write $\Sigma \subseteq \R[x]$ for the cone of sums of squares.
An \emph{SoS certificate} for the nonnegativity of a polynomial $p$ is an explicit representation of $p$ as a sum of squares as in \eqref{SoS}. If such a representation exists, it provides a constructive proof for the nonnegativity of $p$ on $\R^n$.

As observed in \cite{CTR1994}, the problem of finding an SoS certificate for a polynomial $p$ can be formulated as a semidefinite program. Specifically, a polynomial $p \in \mathbb R[x]_{2d}$ of degree $2d$ is a sum of squares if and only if there exists a positive semidefinite matrix $Q$ such that
\begin{equation}\label{SoS-SDP}
    p(x)=[x]_d^T Q[x]_d, 
\end{equation}
where $[x]_d = (x^\alpha)_{|\alpha| \leq d}$ is the vector of all monomials of degree at most $d$. 

The SoS proof system can be extended to verify nonnegativity of polynomials over (basic) \emph{semialgebraic sets}, being subsets of $\R^n$ defined by polynomial inequalities. This permits to define a hierarchy of increasingly accurate approximations for \emph{polynomial optimization problems}, known as the Lasserre hierarchy (or moment-SoS hierarchy)~\cite{Lasserre:lowerbound}. At each level $k$ of the hierarchy, a polynomial optimization problem is approximated using sums of squares of degree up to $2k$. This yields a series of semidefinite programs (of increasing size), each providing a better bound on the minimum $p_{\min}$ of $p$ on the semialgebraic set; namely
\begin{equation}
    \label{EQ:lasserre} \tag{SOS} \max_{\lambda \in \R} \left \{ p - \lambda \text{ has an SoS representation of degree $2k$}\right \} \leq p_{\min}.
\end{equation}

\subsubsection*{Relaxations for copositive programs}

The sum-of-squares hierarchy can be adapted to approximate copositive programs as well. 
Recall that copositivity of a matrix~$M$ can be expressed in terms of nonnegativity of a polynomial associated to $M$, either over $\R^n_+$ or over $\R^n$, see~\eqref{COP1} and~\eqref{COP2}, respectively. The basic idea is to replace the condition of nonnegativity with a {\em sufficient} condition that involves sums of squares. We recall two ways of implementing this idea from the literature.

Based on earlier work of Parrilo~\cite{Parrilo2000StructuredSP}, de Klerk \& Pasechnik~\cite{dKP02} introduced a series of spectrahedral cones approximating the copositive cone, defined as
\begin{equation}\label{eq:def-cones-Kr}
    \mathcal K_n^{\smash{(r)}} := \Big\{M\in \mathcal S^n \colon \big(\sum_{i=1}^n x_i^2\big)^r  (x^{\circ 2})^T M x^{\circ 2} \text{ is a sum of squares}\Big\} \quad (r \in \N).
\end{equation}
In light of~\eqref{COP2}, we have $\mathcal K_n^{\smash{(r)}} \subseteq \COP_n$ for all $r \in \N$. The motivation for this definition is Polya's theorem \cite{pólya1928}: Given a homogeneous polynomial $f\in \mathbb R[x]$ such that $f(x) >0$ for all $x\in\mathbb R_+^n\setminus \{0\}$, there exists $r\in \mathbb{N}$ such that the polynomial $(\sum_{i=1}^n x_i)^r \cdot f$ has nonnegative coefficients; in particular, $(\sum_{i=1}^nx_i^2)^r f(x_1^2, \dots, x_n^2)$ is then a sum of squares (see also Reznick's theorem \cite{Reznick1995}). As a matrix $M\in \MS^n$ belongs to the interior of the copositive cone if and only if $x^TMx>0$ for all $x\in \mathbb{R}_+^n\setminus\{0\}$, we can apply Polya's theorem to find that
\[
    \mathrm{int}(\COP_n) \subseteq \bigcup_{r\geq 0} \MK_n^{\smash{(r)}}.
\]
Zuluaga et al.~\cite{Zuluaga} showed that the cones $\MK_n^{\smash{(r)}}$ can also be formulated as:
\begin{equation}\label{eq: K-ZVP}
    \MK_n^{\smash{(r)}} = \Big \{ M\in \mathcal S^n \colon \big( \sum_{i=1}^n x_i \big)^r x^T M x = \sum_{\substack{\beta \in \mathbb N^n, \\ \vert \beta \vert \leq r+2}} x^{\beta} \sigma_{\beta}, \text{ for some } \sigma_{\beta} \in \Sigma_{r+2-\vert \beta\vert} \Big \}.
\end{equation}
By restricting in \eqref{eq: K-ZVP} to sums of squares of degree $0$ or $2$, they introduced a second series of cones approximating the copositive cone, namely
\begin{equation}\label{eq: def-conesQ}
    \mathcal Q_n^{\smash{(r)}} = \Big \{ M\in \mathcal S^n \colon \big( \sum_{i=1}^n x_i \big)^r x^T M x = \sum_{\substack{\beta \in \mathbb N^n, \\ \vert \beta \vert =r,~r+2}} x^{\beta} \sigma_{\beta}, \text{ for some } \sigma_{\beta} \in \Sigma_{r+2-\vert \beta\vert}\Big \}.
\end{equation}
By definition, we have that $\MQ_n^{\smash{(r)}} \subseteq \MK_n^{\smash{(r)}} \subseteq \COP_n$ for all $r \in \N$. Moreover, using Polya's theorem again, we have that 
\begin{equation} \label{EQ:intCOP_QK}
    \mathrm{int}(\COP_n) \subseteq \bigcup_{r\geq 0} \MQ_n^{\smash{(r)}} \subseteq \bigcup_{r\geq 0} \MK_n^{\smash{(r)}}.
\end{equation}

In this paper, we focus on relaxations of~\eqref{cop-program} based on the inner conic approximations $\MK^{\smash{(r)}}_n$ and $\MQ^{\smash{(r)}}_n$ of the copositive cone $\COP_n$ introduced above. 

\begin{df}
    Given $m+1$ symmetric matrices $A_1, \dots, A_m, C\in \MS^n$, and a vector $b\in \mathbb{R}^m$, we define the following conic optimization programs
\begin{align}
    p_{\MK}^{\smash{(r)}} &= \min_{y\in \mathbb{R}^m}\Big\{ b^Ty: \sum_{i=1}^my_iA_i-C \in \MK_n^{\smash{(r)}}\Big\}, \label{eq:conic-prog1} \tag{CP-K} \\
    p_{\MQ}^{\smash{(r)}} &= \min_{y\in \mathbb{R}^m}\Big\{ b^Ty: \sum_{i=1}^my_iA_i-C \in \MQ_n^{\smash{(r)}}\Big\}. \label{eq:conic-prog2} \tag{CP-Q}
\end{align}
\end{df}
Let $p^*$ be the optimum solution value of the copositive program~\eqref{cop-program}. Since $\MQ_n^{\smash{(r)}} \subseteq \MK_n^{\smash{(r)}} \subseteq \COP_n$ for all $r \in \N$, we see immediately that 
\[
    p_{\MQ}^{\smash{(r)}} \geq p_{\MK}^{\smash{(r)}} \geq p^* \quad (\forall r \in \N).
\]
Moreover, assuming that~\eqref{cop-program} is strictly feasible, and in light of~\eqref{EQ:intCOP_QK}, we have that
\[
     \lim_{r \to \infty } p_{\MK}^{\smash{(r)}} = \lim_{r \to \infty } p_{\MQ}^{\smash{(r)}} = p^{*}.
\]
The parameters $p_{\MK}^{\smash{(r)}}$ and $p_{\MQ}^{\smash{(r)}}$ thus provide two monotonely decreasing sequences of upper bound on $p^*$, which are convergent if~\eqref{cop-program} is strictly feasible. The cones $\MK_n^{\smash{(r)}}$  (resp. $\MQ_n^{\smash{(r)}}$) and the bounds $p_{\MK}^{\smash{(r)}}$ (resp. $p_{\MQ}^{\smash{(r)}}$) have been intensively studied in the literature in several contexts. We refer to the works \cite{DobreVera,Rendl,DDGH,SV,LV23} and the references therein.

It is possible to provide other sequences of conic approximations restricting relation \eqref{COP1} to the standard simplex $\Delta_n$ and \eqref{COP2} to the unit sphere $\mathbb S^{n-1}$ (see \cite{LAURENT202226}), but these are beyond the scope of this work.

\subsection{Our contributions}\label{subsection:our contribution}

Our main technical contribution is the following theorem, which shows that the relaxations \eqref{eq:conic-prog1} and~\eqref{eq:conic-prog2} of a copositive program~\eqref{cop-program} can be computed in polynomial time (up to a small additive error) under two technical assumptions.
\begin{thm} \label{THM:main}
Let $C, A_1, \ldots, A_m$ be the matrices defining the copositive program~\eqref{cop-program}.
Suppose that the following assumptions hold:
    \begin{enumerate}
        \item[(PBOS)] The program \eqref{eq:conic-prog1} (resp.~\eqref{eq:conic-prog2}) has a \emph{polynomially bounded optimal solution (PBOS)}. That is, an optimal solution $y^*\in\mathbb [-R, R]^m$, with $R \leq 2^{\poly(n)}$. 
        \item[(intSPN)] There exists a vector $\bar y\in\mathbb Q^m$, a nonnegative matrix $N\in \mathcal{N}_n$, and a positive definite matrix $P\in \MS^n_+$ with $\lambda_{\min}(P)\geq 2^{-\poly(n)}$, all encodable in $\poly(n)$ bits, such that,
        \[
            \sum_{i=1}^m\bar y_i A_i-C =P+N.
        \]
        That is, \eqref{cop-program} has a sufficiently nice feasible point in $\SPN_n = \MS_n^+ + \mathcal{N}_n$.
    \end{enumerate}
    Then, for any fixed $r \in \N$, the parameter $p_\MK^{\smash{(r)}}$ (resp. $p_\MQ^{\smash{(r)}}$) can be computed up to an additive error $\ep > 0$ in time $\poly(n, \log(1/\ep$)).
\end{thm}

We note that the cones $\MK_n^{\smash{(0)}}$ and $\MQ_n^{\smash{(0)}}$ are both equal to $\SPN_n$. Thus, condition~\eqref{intSPN} may be viewed as a strict feasibility condition on \eqref{eq:conic-prog1}, \eqref{eq:conic-prog2}.

\subsubsection*{Standard quadratic programs and the weighted stability number} 

As our second contribution, we show that the conditions of~\Cref{THM:main} are satisfied for several notable examples of copositive programs.
Let $M\in \MS^n$ be a symmetric matrix. A \emph{standard quadratic optimization problem (SQP)} is a minimization problem over the standard simplex $\Delta_n:=\{x\in \mathbb{R}^n: \sum_{i=1}^nx_i=1, \ x_i\geq 0 \text{ for }i\in [n]\}$ of the following form: 
\begin{align} \label{EQ:SQP} \tag{SQP}
    p_{\min} =\min \Big\{x^TMx: x\in \Delta_n\Big\}.
\end{align}
An SQP can be reformulated as a copositive program~\cite{Bomze2000OnCP, dKP02, Burer09}, see~\eqref{COP-SQ1} in \Cref{section:SQP}. Using~\Cref{THM:main}, we show the following.
\begin{thm} \label{PROP:main_SQP}
    For any $r \in \N$ fixed, the parameters $p_\MK^{\smash{(r)}}$ and $p_{\MQ}^{\smash{(r)}}$ corresponding to the copositive formulation~\eqref{COP-SQ1} of a standard quadratic program can be computed up to an additive error $\ep > 0$ in time $\poly(n, \log (1/\ep))$. 
\end{thm}

Additionally, if $p_{\min}>0$ in~\eqref{EQ:SQP}, then $1/{p_{\min}}$ can be formulated as a copositive program as well, see~\eqref{COP-INVERSE1} in  \Cref{section:SQP}. 
Under a mild assumption on the matrix $M$, we show that the bounds \eqref{eq:conic-prog1} and~\eqref{eq:conic-prog2} corresponding to~\eqref{COP-INVERSE1} can be computed in polynomial time (up to a small additive error), see \Cref{THM:COPINV}.
As a special case, this includes the relaxations of the copositive program~\eqref{alpha} that computes the (weighted) stability number $\alpha(G)$ of a graph. Note that, while we could have applied \Cref{PROP:main_SQP} to the SQP~\eqref{EQ:alpha_SQP}, this would yield bounds on $1/\alpha(G)$, which are generally less useful than bounds on~$\alpha(G)$.

\subsubsection*{Bounded copositive programs and the chromatic number}

A limitation of~\Cref{THM:main} is that, even if the original copositive program~\eqref{cop-program} satisfies \eqref{PBOS}, it is not clear whether the relaxations~\eqref{eq:conic-prog1}, \eqref{eq:conic-prog2} do as well. This means it is typically not  easy to verify whether the first condition of~\Cref{THM:main} holds. 
However, as we show in~\Cref{section:bounded-CP}, as long as \eqref{PBOS} holds for a copositive program~\eqref{cop-program}, it is possible to define a modified program~\eqref{CP-1} whose optimal solution value is equal, but whose relaxations~\eqref{eq:conic-prog1}, \eqref{eq:conic-prog2} are guaranteed to verify \eqref{PBOS}. Furthermore, if \eqref{cop-program} satisfies the second condition of~\Cref{THM:main}, then~\eqref{CP-1} does as well. This leads to our third contribution:
\begin{thm} \label{THM:main2}
    Let $r \in \mathbb N$ fixed. Suppose that the following assumptions hold:
    \begin{itemize}
        \item The program~\eqref{cop-program} has a polynomially bounded optimal solution;
        \item The program~\eqref{cop-program} satisfies~\eqref{intSPN}. 
    \end{itemize}
    Then, there exists a copositive program \eqref{CP-1} whose optimal solution value is equal to that of~\eqref{cop-program}, and which satisfies the conditions of \Cref{THM:main}. In particular, the relaxations~\eqref{eq:conic-prog1}, \eqref{eq:conic-prog2} of~\eqref{CP-1} can be computed up to an additive error~$\ep > 0$ in time $\poly(n, \log(1/\ep$)).  
\end{thm}
An important application of~\Cref{THM:main2} is to 
the \emph{chromatic number} $\chi(G)$ of a graph $G$. 
The chromatic number is the minimum number of colors required to color the vertices of $G$ in such a way that adjacent vertices have different colors. Computing $\chi(G)$ is a classical NP-hard problem. It can be formulated as a copositive program~\cite{GvozdenovicLaurent}, see~\eqref{chromatic-number-cop} below. While it is not clear whether~\Cref{THM:main} applies to this program, we show in \Cref{section:bounded-CP} that~\Cref{THM:main2} does.

\subsubsection*{Necessity of our assumptions}

As our final contribution, we provide some illustrative pathological examples of copositive programs in \Cref{section:examples}. In particular,
\begin{itemize}
    \item \Cref{ex:khachiyan-type} provides a copositive program of small bit size whose feasible solutions all contain entries of doubly-exponential size (in $n$). This example generalizes a construction of Khachiyan for semidefinite programs, see ~\Cref{SEC:related-word} below. This example indicates that a boundedness assumption on the feasible region of~\eqref{cop-program} is needed.
    \item \Cref{ex:SDPnotfeasible} provides a feasible copositive program whose semidefinite relaxations \eqref{eq:conic-prog1} and \eqref{eq:conic-prog2} are not feasible for any~${r \in \N}$. This shows that some sort of strict feasibility condition on~\eqref{cop-program} is required. 
    \item \Cref{ex:PBOSSDPneeded} provides a copositive program that satisfies conditions~\eqref{PBOS} and \eqref{intSPN}, and whose semidefinite relaxations \eqref{eq:conic-prog1} and \eqref{eq:conic-prog2} are feasible, but do \emph{not} satisfy~\eqref{PBOS}. In particular, these relaxations cannot be computed in polynomial time by traditional means. This example shows that the first condition of \Cref{THM:main} is necessary in the sense that it cannot be replaced by the first condition of \Cref{THM:main2} in general.
\end{itemize}

\subsection{Technical overview}

We briefly describe the central ingredients of the proof of our main result \Cref{THM:main}, which shows that the relaxations~\eqref{eq:conic-prog1} and \eqref{eq:conic-prog2} obtained by replacing the copositive cone in~\eqref{cop-program} by the cones~$\MK_n^{\smash{(r)}}$ and $\MQ_n^{\smash{(r)}}$, respectively, can (approximately) be computed in polynomial time. 

First, we leverage the SoS structure of these cones to reformulate \eqref{eq:conic-prog1} and \eqref{eq:conic-prog2} as semidefinite programs (SDPs) in standard form. Importantly, under conditions \eqref{PBOS} and \eqref{intSPN}, the bit size of these semidefinite programs is polynomial in the input size of the original program~\eqref{cop-program}.

Next, we show that these SDPs can be solved  in polynomial time (up to arbitrary fixed precision). For this, we will use~\Cref{thm:approx sol}, due to de Klerk \& Vallentin~\cite{dKV16}. In order to apply their result, we have to exhibit a feasible solution $X_0$ to the SDP formulations of small bit size which satisfies
\[
    B(X_0,R_1) \subseteq \mathcal R \subseteq B(X_0,R_2),
\]
where $\mathcal R$ is the feasible region of the SDP, and $R_1, R_2 \in \Q$ have small bit size. That is to say, we must show the feasible regions of these SDPs contain a ball which is not too small, and are contained in a ball which is not too large. We show that the former requirement follows from condition~\eqref{intSPN}. Indeed, starting from the $n\times n$ matrix given by property (\ref{intSPN}) we construct the matrix $X_0$ (of an appropriate size) and find $R_1\in \mathbb{Q}$ satisfying the conditions. On the other hand, the latter does not follow easily from~\eqref{PBOS}. Indeed, even if the conic formulations~\eqref{eq:conic-prog1}, \eqref{eq:conic-prog2} have a polynomially bounded feasible solution, this does not automatically mean their SDP formulations do as well.

The key point is that we have to control the magnitude of the coefficients that appear in the sum-of-squares proofs of membership in the cones $\MK_n^{\smash{(r)}}$ and $\MQ_n^{\smash{(r)}}$ on which the SDP formulations rely. To do so, we prove a general result (\Cref{thm:SoS bound coeff-1}) bounding the coefficients of certain sum-of-squares proofs, which we believe to be of independent interest. We then apply this result to obtain bounds in our setting (see \Cref{coro-SoS}).
 
To prove~\Cref{thm:SoS bound coeff-1}, we first derive an upper bound on the size of the coefficients of a polynomial in terms of its supremum norm over an appropriate subset of~ $\R^n$ (see \Cref{lemma-use}). By applying this bound to `both sides' of a sum-of-squares proof, we are able to obtain a bound on the coefficients that appear on the `SoS side' of the proof.

We derive our remaining contributions from \Cref{THM:main}; either by verifying conditions~\eqref{PBOS} and~\eqref{intSPN} directly for specific classes of copositive programs (see \Cref{section:SQP}), or by modifying copositive programs so that these conditions are guaranteed to hold (see \Cref{section:bounded-CP}). 

Finally, to construct our pathological examples in~\Cref{section:examples}, we combine an existing construction of Khachiyan for semidefinite programs with copositive programs arising from specific graphs, including the cycle $C_5$.

\subsection{Related work}\label{SEC:related-word}

This work is closely related the following fundamental open question on the complexity of semidefinite programming: \emph{Can the feasibility of an SDP with well-behaved coefficients be decided in polynomial time?} 
The primary obstacle to resolving this question is the fact that there exist SDPs whose feasible solutions all have exponential bit size, and thus cannot even be written down in polynomial time. (In contrast, \emph{linear} programs with well-behaved coefficients always admit well-behaved feasible solutions). 
These examples go back to work of Khachiyan, see~\cite{Pataki2024} for a discussion.

Recently, O'Donnell~\cite{odonnell:LIPIcs.ITCS.2017.59} showed that such examples can even arise from the sum-of-squares hierarchy for polynomial optimization. In fact, this may already happen for optimization over the binary cube $\{-1, 1\}^n$. 
A common (but erroneous) belief prior to his work was that a standard compactness assumption for polynomial optimization problems (namely, \emph{explicit boundedness}\footnote{Explicit boundedness is a slight strengthening of the \emph{Archimedean condition}, see \cite[Rem. 3]{GPS23}.}) would preclude such pathologies.

Subsequently, Raghavendra \& Weitz~\cite{RaghavendraWeitz2017} and Gribling, Polak \& Slot~\cite{GPS23} gave sufficient conditions for polynomial optimization problems to guarantee that their SoS relaxations can be computed in polynomial time (up to a small additive error). 
Together, these conditions cover many of the problems considered in the theoretical computer science and optimization communities (see \cite{BharathiM21,BharathiM22,BulatovRSTOC22,BulatovRSTACS22,Mastrolilli21TALG} for recent applications). However, neither apply to the SoS relaxations of copositive programs that we consider in this paper. To the best of the authors' knowledge, this work is the first to establish conditions under which polynomial-time computability of the relaxations~\eqref{eq:conic-prog1}, \eqref{eq:conic-prog2} of a copositive program~\eqref{cop-program} can be guaranteed.

We briefly compare the conditions of~\Cref{THM:main} to those of~\cite{RaghavendraWeitz2017, GPS23}. Condition~\eqref{PBOS} is closely related to explicit boundedness: both guarantee the existence of bounded optimal solutions, respectively to \eqref{eq:conic-prog1}, \eqref{eq:conic-prog2} and \eqref{EQ:lasserre}. (For both copositive programs and polynomial optimization problems, boundedness of the original problem does not guarantee boundedness of the relaxations.) An important difference however, is that explicit boundedness can be readily verified, whereas in general, condition~\eqref{PBOS} cannot. This limitation is addressed, to some extent, by \Cref{THM:main2}. 
Condition~\eqref{intSPN} is similar to strict feasibility conditions imposed in~\cite{RaghavendraWeitz2017, GPS23}, see e.g., Theorem 5 and Proposition 6 in~\cite{GPS23}. It is unclear how to verify these types of conditions in general, and both the present work and~\cite{RaghavendraWeitz2017, GPS23} are able to do so only in specific instances.

\section{Preliminaries}

We recall some definitions and results that will be used in the rest of the paper.

\subsection{Notations}
The set of real symmetric matrices of dimension $n\times n$ will be denoted by $\mathcal{S}^n$. The space $\mathcal S^n$ is equipped with the Frobenius norm, defined as $\Vert X \Vert_F=\sqrt{\langle X,X\rangle}$, where the inner product is given by $\langle X,Y\rangle:=\mathrm{Tr}(XY)$, with $\mathrm{Tr}$ the trace operator.
For a matrix $A\in \mathcal S^n$, we write $A\succeq 0$ to indicate that $A$ is positive semidefinite, and $A\succ 0$ to indicate that $A$ is positive definite. The set of positive semidefinite matrices of dimension $n\times n$ is denoted by $\mathcal S^+_n$. The minimum eigenvalue of $A$ is denoted by $\lambda_{\min}(A)$. Additionally, we write $A\geq 0$ to denote that $A$ has nonnegative entries. The cone of $n\times n$ symmetric matrices with nonnegative entries is denoted by $\mathcal{N}_n$.  
Given scalars $b_1,\dots,b_n$, we denote by $\mathrm{Diag}(b_1,\dots,b_n)$ the $n\times n$ diagonal matrix with diagonal elements $b_1,\dots,b_n$. The direct sum of the matrices $A_1,\dots,A_n$ is denoted by $A_1\oplus\dots\oplus A_n$.
For a multi-index $\alpha\in \mathbb{N}^n$, we let $|\alpha|:= \| \alpha\|_1=\sum_{i=1}^n\alpha_i$.
We denote by $\R[x]$ the space of $n$-variate, real polynomials. For $x \in \R^n$, $d \in \N$, we write $[x]_d = (x^\alpha)_{|\alpha| \leq d}$ for the vector of all monomials of degree at most $d$.

\subsection{Bit size} \label{SEC:bitsize}
In the following discussion, we will frequently refer to the \emph{bit size} of the coefficients of polynomials and other input elements. The bit size of an integer $z\in \mathbb Z$ is defined as the number $\size(z) := 1+ \lceil \log_2(|z|+1) \rceil$ of bits required to encode $z$ as a binary string.
For a rational number $q=\frac{q_1}{q_2}$, the bit size is defined as $\size(q) := 1+ \lceil \log_2(|q_1|+1) \rceil + \lceil \log_2(|q_2|+1) \rceil$ accounting for the bit size of both the numerator $q_1 $and denominator $q_2$. The bit size of an $m\times n$ matrix $A$ is given by $\size(A)= mn + \sum_{i,j} \size(A_{ij})$, see also \cite{Schrjiver-ilp_theory}. 
More generally, the bit size of a mathematical program is the cumulative number of bits necessary to encode all its input elements as binary strings.
Throughout, we will assume that the copositive programs~\eqref{cop-program} we consider have polynomial bit size in~$n$. In particular, the entries of the matrices defining these programs are assumed to be rational and encodable in a number of bits polynomial in $n$.

\subsection{Semidefinite programming} 
Let $C\in \mathbb Q^{n\times n}$, $A_i\in \mathcal S^n$ and $b_i \in \mathbb Q$ for $1\leq i\leq m$ and consider the following semidefinite program in standard form:
\begin{equation}\label{eq:SDP} \tag{SDP}
    \mathrm{val} = \inf \{\langle C,X\rangle : \langle A_i,X\rangle= b_i,\ i=1,\dots,m,\ X\succeq 0\},
\end{equation}
with feasible region
\begin{equation*}
    \mathcal R := \{X\succeq 0 :\ \langle A_i,X\rangle = b_i,\ i=1,\dots,m\}.
\end{equation*}
Under certain assumptions, it is possible to solve the program \eqref{eq:SDP} in polynomial time (up to a small additive error) as shown in \cite{Grtschel1981TheEM} and \cite{dKV16} (using the Ellipsoid algorithm and interior point methods, respectively). We will rely on this fact to show polynomial time computability for the relaxations~\eqref{eq:conic-prog1} and~\eqref{eq:conic-prog2} in~\Cref{section:main-results} below.
We use the following formulation of the main result of~\cite{dKV16}:

\begin{thm}[{\cite[Thm. 1.1]{dKV16}, see also \cite[Thm. 15]{GPS23}}] \label{thm:approx sol}
Let $R_1,R_2 >0$ be given and suppose that there exists an $X_0\in \mathcal R$ so that:
\begin{equation*}
    B(X_0,R_1) \subseteq \mathcal R \subseteq B(X_0,R_2),
\end{equation*}
where $B(X_0,\lambda)$ is the ball of radius $\lambda\in \mathbb R$ (in the Frobenius norm $\Vert \cdot \Vert_F$) centered at $X_0$ in the subspace
\begin{equation*}
    V = V(\mathcal R) := \{X\in \mathcal S^n : \langle A_i,X\rangle = b_i, \ 1\leq i\leq m\}.
\end{equation*}
Then for any rational $\varepsilon>0$ one can find a rational matrix $X^*\in\mathcal R$ such that
\begin{equation*}
    \langle C,X^*\rangle - \mathrm{val} \leq \varepsilon
\end{equation*}
in time polynomial in $n,m,\log(R_2/R_1),\log(1/\varepsilon)$ and the bit size of the input data $C,A_i,b_i$ and the feasible point $X_0$. 
\end{thm}

\subsection{Supremum norm and coefficient norm}\label{SEC:norms}
Let $p=\sum_{|\alpha|\leq d} c_\alpha x^\alpha$ be a multivariate polynomial of degree $d\in \mathbb N$. The \emph{coefficient norm} of $p$ is defined as 
    $$\|p\|_{\mathbb{R}[x]} = \max_{\alpha} \frac{|c_\alpha|}{\binom{|\alpha|}{\alpha}},$$
where the multinomial coefficient $\binom{|\alpha|}{\alpha}$ is given by 
    $$\binom{|\alpha|}{\alpha} = \frac{|\alpha|!}{\alpha_1!\cdot\alpha_2!\cdot\dots \cdot\alpha_n!}.$$
The coefficient norm of a polynomial can be bounded in terms of its supremum norm on $[-1, 1]^n$ as follows.
\begin{lemma}[\cite{KORDA20171}] \label{lemma-corbi}
    Let $p\in \mathbb{R}[x]_d$, then 
        $$\|p\|_{\mathbb{R}[x]} \leq 3^{d+1}\max_{x\in [-1,1]^n}|p(x)|.$$
\end{lemma}

\Cref{lemma-corbi} allows us to upper bound the size of the largest coefficient of a polynomial in terms of its supremum norm on a small ball centered at the origin. This will be useful in~\Cref{section:bounded-coeff}, where we analyze the coefficients of sum-of-squares decompositions of certain well-behaved polynomials.
\begin{lemma}\label{lemma-use}
Let $p=\sum_{|\alpha|\leq d}c_\alpha x^\alpha$, and let $0<r<n$ be a real number. Then, we have
  $$\max_{|\alpha|\leq d} |c_\alpha| \leq 3^{d+1}d! \Big(\frac{n}{r}\Big)^{\frac{d}{2}}\max_{\sum_{i=1}^nx_i^2\leq r}|p(x)|.$$  
\end{lemma}
\begin{proof}
    Observe that Lemma \ref{lemma-corbi} implies 
        $$\max_{|\alpha|\leq d} |c_\alpha| \leq 3^{d+1}d! \max_{x\in [-1,1]^n}|p(x)|,$$
    and moreover 
    \begin{align}\label{ineq-use}
        \max_{|\alpha|\leq d} |c_\alpha| \leq 3^{d+1}d! \max_{\sum_{i=1}^nx_i^2\leq n}|p(x)|.
    \end{align}
    Given a real number $0<a<1$, we consider the polynomial $q(x):= p(a\cdot x)$ obtained by scaling the variables by a factor of $a$. Thus, we have
        $$q = \sum_{|\alpha|\leq d} a^{|\alpha|}c_\alpha x^\alpha.$$
    By applying inequality (\ref{ineq-use}) to polynomial $q$, we obtain 
        $$\max_{|\alpha|\leq d} |a^{|\alpha|}c_\alpha| \leq 3^{d+1}d! \max_{\sum_{i=1}^nx_i^2\leq n}|p(a\cdot x)|.$$
    By substituting $y_i=ax_i$, and using $a^d \leq a^{|\alpha|}$, we get the following inequality:
        $$a^d\max_{|\alpha|\leq d}|c_\alpha| \leq 3^{d+1}d! \max_{\sum_{i=1}^ny_i^2\leq na^2}|p(y)|.$$
    Finally, we set $a=\sqrt{\frac{r}{n}}$ and obtain the desired result.
\end{proof}

\section{Bounding the coefficients in sum-of-squares proofs} \label{section:bounded-coeff}

In this section, we present a  sufficient condition that guarantees that the coefficients in certain sum-of-squares proofs of nonnegativity have polynomial bit size in $n$. We apply this condition in~\Cref{section:main-results} to prove our main result. We first recall a useful result from~\cite{GPS23}:

\begin{lemma}\label{lemma-ball}
    Let $K=\{x\in \mathbb{R}^n: g_1(x)\geq 0, \dots, g_m(x)\geq 0\}$ be a full-dimensional semialgebraic set, defined by polynomials $g_1, \ldots, g_m$ whose coefficients have polynomial bit size in $n$. Assume there exists $x^* \in \R^n$, with $\|x^*\|_2 \leq 2^{\poly(n)}$ such that $g_i(x^*)>0$ for all $i=1, \dots, m$. Then, there exists $r \in \R$ with $2^{-\poly(n)} \leq r \leq 2^{\poly(n)}$ such that 
    \[
    B_r(x^*)= \{y\in \mathbb{R}^n: \|y-x^*\|_2\leq r\}\subseteq K.
    \]
\end{lemma}

Combining \Cref{lemma-ball} with \Cref{lemma-use}, we obtain the following.
\begin{thm}\label{thm:SoS bound coeff-1}
    Let $f, g_1, \dots, g_m\in \mathbb{R}[x]$ be polynomials whose coefficients are rational and can be encoded with $\poly(n)$ bits. Let
        $$K=\{x\in \mathbb{R}^n: g_i(x)\geq 0 \text{ for } i=1, \dots, m\}.$$
    Assume there exists $z\in K$ a rational point with $\size(z)\leq \poly(n)$ such that $g_i(z)>0$ for all $i=1, \dots, m$. Let $d\in \mathbb{N}$ be a constant, and assume there exists a degree $d$ proof of nonnegativity of $f$ over $K$ 
        $$f= \sum_{j=1}^{s_1} p_j^2 + \sum_{i=1}^{m} g_i\cdot\Big(\sum_{j=1}^{s_i} q_{j,i}^2\Big).$$
    Then,  the coefficients of $p_j, q_{i,j}$ are bounded by $2^{\poly(n)}$.
\end{thm}
\begin{proof}
    Let $\varepsilon= \min\{ g_i(z): i=1,\dots, m\}$. Observe that $\varepsilon$ is positive, rational and $\size(\varepsilon)={\poly(n)}$. This implies that $\varepsilon\geq 2^{-\poly(n)}$. We consider the semialgebraic set 
        $$K'=\{x\in \mathbb{R}^n: g_i(x)\geq \frac{\varepsilon}{2} \text{ for } i=1, \dots, m\}.$$
    By Lemma \ref{lemma-ball}, there exists $2^{-\poly(n)} \leq r \leq 2^{\poly(n)}$ such that $B_r(z)\subseteq K'$. Now we take the maximum of $f$ over the ball $B_r(z)$.
    \begin{align*}
        \max_{x\in B_r(z)} f(x) &= \max_{x\in B_r(z)} \Big(\sum_{j=1}^{s_1}p_j(x)^2+ \sum_{i=1}^m g_i(x)\sum_{j=1}^{s_i}q_{i,j}(x)^2\Big)\\
        &\geq \max_{x\in B_r(z)} \Big(\sum_{j=1}^{s_1}p_j(x)^2+ \frac{\varepsilon}{2}\sum_{i=1}^m \sum_{j=1}^{s_i}q_{i,j}(x)^2\Big).
    \end{align*}
    Since all terms in the right hand side are nonnegative, and the maximum of $f$ over $B_r(z)$ is bounded from above by $2^{\poly(n)}$, we obtain
        $$\max_{x\in B_r(z)} p_j(x)^2 \leq 2^{\poly(n)} \text{ for } i=1, \dots, s_1,$$
    and thus, 
        $$\max_{x\in B_r(z)} |p_j(x)| \leq 2^{\poly(n)} \text{ for } i=1, \dots, s_1.$$
    Similarly, since $\varepsilon> 2^{-\poly(n)}$, we obtain 
        $$\max_{x\in B_r(z)}  q_{i,j}(x)^2 \leq 2^{\poly(n)} \text{ for } i=1, \dots, m, \quad j=1, \dots, s_2,$$
    so that, 
        $$\max_{x\in B_r(z)}  |q_{i,j}(x)| \leq 2^{\poly(n)} \text{ for } i=1, \dots, m, \quad j=1, \dots, s_2.$$
    
    Then, it suffices to prove that the coefficients of $p_1$ are bounded by $2^{\poly(n)}$. 
    We define $\tilde{p_1}(x):= p_1(x-z)$. Then, we have 
        $$\max_{\sum_{i=1}^n x_i^2\leq r} |\tilde{p_1}(x)|\leq 2^{\poly(n)}.$$
    Let $\tilde{p_1}=\sum_{|\alpha|\leq d} \tilde{c_\alpha}x^{\alpha}$. Combining Lemma \ref{lemma-use} and the previous inequality, we obtain 
        $$\max |\tilde{c_\alpha}| \leq 3^{d+1}d! \Big(\frac{n}{r}\Big)^{\frac{d}{2}}\max_{\sum_{i=1}^nx_i^2\leq r}\tilde{p_1}(x) \leq 2^{\poly(n)}.$$ 
    
    This shows that the coefficients of $p_1$ are bounded by $2^{\poly(n)}$ because we have $p_1(x)= \tilde{p_1}(x+z)$, and the entries of $z$ are bounded by $2^{\poly(n)}$. 
\end{proof}

From Theorem \ref{thm:SoS bound coeff-1} we can directly derive:
\begin{cor}\label{coro-SoS}
    Let $M\in \MS^n$ be a symmetric matrix such that $|M_{ij}|\leq 2^{\poly(n)}$, and let $r\in \mathbb{N}$ be a constant.
    \begin{enumerate}
        \item  If $M\in \MK_n^{\smash{(r)}}$, then the coefficients of the sums of squares in the representations as in definition (\ref{eq:def-cones-Kr}) and (\ref{eq: K-ZVP}) are bounded by $2^{\poly(n)}$.
        \item If $M\in \MQ_n^{\smash{(r)}}$, then the coefficients of the sums of squares in the representation as in (\ref{eq: def-conesQ}) are bounded by $2^{\poly(n)}$.  
    \end{enumerate}
\end{cor}
\begin{proof}
    We apply Theorem \ref{thm:SoS bound coeff-1} in several settings. For part (1) and the representation as in (\ref{eq:def-cones-Kr}) the result follows after setting $f=(\sum_{i=1}^nx_i^2)^r(x^{\circ2})^TMx^{\circ2}$, $K=\{x\in \mathbb{R}: 1\geq 0\}$, and $z=(1, \dots, 1)$. For part 1) and the representation as in (\ref{eq: def-conesQ}) the result follows after setting $f=(\sum_{i=1}^nx_i)^rx^TMx$,
        $$K=\{x\in \mathbb{R}^n: x^\beta \geq 0, \text{ for } \beta\in \mathbb{N}^n, |\beta|\leq r, |\beta|\equiv r \ (\text{mod }2)\},$$
    and $z=(1, \dots, 1)$. Finally, for part (2) we set $f=(\sum_{i=1}^nx_i)^rx^TMx$,
        $$K=\{x\in \mathbb{R}^n: x^\beta \geq 0, \text{ for } \beta\in \mathbb{N}^n, |\beta|=r,r+2\},$$
    and $z=(1,\dots, 1)$.
\end{proof}

\section{Proof of \Cref{THM:main}}\label{section:main-results}

Let $A_1, \dots, A_m, C\in \MS^n$ be symmetric matrices and let $b\in \mathbb{R}^m$ be a vector. Let $\mathcal C$ be a matrix cone. We consider programs of the form:
\begin{align}\label{eq:conic-prog}
    p = \min_{y\in \mathbb{R}^m}\Big\{ b^Ty: \sum_{i=1}^my_iA_i-C \in \mathcal C\Big\}.
\end{align}
Recall the two properties for this type of program that feature in \Cref{THM:main}.
\begin{df}\label{def:properties}
Consider a conic program as in \eqref{eq:conic-prog}. We say that 
    \begin{enumerate}[label=PBOS]
        \item \label{PBOS} Program \eqref{eq:conic-prog} has a \emph{polynomially bounded optimal solution}, (\emph{PBOS}) for short, if there exist an $R \leq 2^{\poly(n)}$ and an optimal solution $y^*\in\mathbb R^m$ for \eqref{eq:conic-prog} such that $y^*\in [-R,R]^m$; 
    \end{enumerate}
    \begin{enumerate}[label=intSPN]
        \item \label{intSPN} The tuple $(A_1,\dots,A_m,C)$ has an \emph{interior $\SPN^{\rm low}_n$ point}, (\emph{intSPN}) for short, if there exists $\bar y\in\mathbb Q^m$ with $\size(\bar y) \leq \poly(n)$ such that 
        \[
            \sum_{i=1}^m\bar y_i A_i-C \in \SPN^{\rm low}_n,
        \]
        where the set $\SPN^{\rm low}_n$ is defined as
        \begin{align*}
            \SPN^{\rm low}_n := \{M\in \mathcal S^n \colon &M=P+N \text{ for } N\geq 0 \text{ and } P\succ 0 \\
            &\text{ such that } \lambda_{\min}(P)\geq 2^{-\poly(n)} \text{ and } \\ 
            &P_{ij}, N_{ij}\text{ have bit size}\leq \poly(n) \ \forall i,j=[n]\}.
        \end{align*}
    \end{enumerate}
\end{df}
\begin{obs}\label{obs:R}
    Notice that if both properties \eqref{PBOS} and \eqref{intSPN} hold for program \eqref{eq:conic-prog}, i.e. there exists a polynomially bounded optimal solution $y^*$ and an interior $\SPN^{\rm low}_n$ point $\bar y$, then we can find a natural number $R \leq 2^{\poly(n)}$ such that 
    \[y^*,\bar y \in [-R,R]^m.\]
\end{obs}

Below, we consider the programs~\eqref{eq:conic-prog1} and \eqref{eq:conic-prog2} separately. In both cases, we show that if they satisfy conditions \eqref{PBOS} and~\eqref{intSPN}, their optimal values can be computed in polynomial time (up to fixed precision).

\subsection{Polynomial-time computation of \texorpdfstring{\eqref{eq:conic-prog1}}{CP-K}}
Recall that~\eqref{eq:conic-prog1} is obtained by replacing the copositive cone $\COP_n$ in~\eqref{cop-program} by the cones~$\MK_n^{\smash{(r)}}$ defined in~\eqref{eq:def-cones-Kr}, i.e.,
\begin{equation} \tag{\ref{eq:conic-prog1}}\label{cop-program-approx1}
    p^{\smash{(r)}}_{\MK} := \min\Big\{ b^Ty: \sum_{i=1}^m y_i A_i-C \in \MK_n^{\smash{(r)}}\Big\}.
\end{equation}
In order to prove~\Cref{THM:main} for this relaxation, we first formulate it as a semidefinite progam. Suppose that~\eqref{cop-program-approx1} satisfies~\eqref{PBOS} for some $R \leq 2^{\poly(n)}$ and that it satisfies~\eqref{intSPN} with feasible point $\bar y \in [-R, R]^m$. Then, \eqref{cop-program-approx1} is equivalent to the following semidefinite program:
\begin{equation}\label{eq:ref stqp1}
    \min \Big \{\Big\langle \tilde Q, 
    \left(\begin{array}{c|c} 
    0 & 0\\
    \hline 
    0 & B 
    \end{array}\right)\Big\rangle \colon
    \tilde Q = \left(\begin{array}{c|c} 
    Q & 0\\
    \hline 
    0 & D  
    \end{array}\right), 
    \langle \tilde Q, A_\gamma\rangle = b_\gamma, \ \tilde Q\succeq 0\Big \}.
\end{equation}
Here, $B\in \mathcal S^{2m}$ is the diagonal matrix $\mathrm{Diag}(b_1/2,-b_1/2,\dots,b_m/2,-b_m/2)$, and the matrix $D\in \mathcal S^{2m}$ models a $2m\times 2m$ diagonal matrix encoding the variables $y_i$ in the following way: we impose linear constraints so that the matrix $D$ takes the shape $\mathrm{Diag}(d_1^+,d_1^-,\dots,d_m^+,d_m^-)$.
For each $i\in [m]$, the variable $d_i^+$ models $2R+y_i$, and we impose the linear constraints
    $$d_i^- = 4R - d_{i}^+.$$
Hence, for each $i\in [m]$, the variable $d_i^-$ models $2R-y_i$.

The shape of $\tilde Q$ is imposed via a polynomial number of linear constraints. For $p_y := (x^{\circ 2})^T(\sum_{i=1}^m y_i A_i - C)x^{\circ 2}$, we write  
    $$(\sum_{i=1}^n x_i^2)^rp_y=\sum_{|\gamma|=2r+4}{c_\gamma}(y)x^{\gamma},$$
where $c_\gamma(y)$ is a linear function on $y$ whose coefficients are bounded by~$2^{\poly(n)}$. Then, we can write $c_\gamma(y)=\langle D, C_\gamma \rangle$ for a matrix $C_\gamma$ with rational entries bounded by $2^{\poly(n)}$. The matrix $Q$ is such that  we have the equality  
    $$(\sum_{i=1}^n x_i^2)^r p_y = [x]_{r+2}^T Q [x]_{r+2}, \text{ and } Q\succeq 0,$$ 
thus by equating the coefficients we have $\langle Q, \tilde{A}_{\gamma}\rangle = c_\gamma(y)$ for some matrices $\tilde{A}_\gamma$ with entries bounded by $2^{\poly(n)}$. Then, each of these constraints can be written as $\langle Q, \tilde{A_\gamma}\rangle = \langle D, C_\gamma \rangle$, which corresponds to the constraint of the form $\langle \tilde Q, A_\gamma\rangle = b_\gamma$.

\begin{thm}[Special case of \Cref{THM:main}] \label{thm:cones-K}
    Let $r\in \mathbb{N}$. Assume that \eqref{cop-program-approx1} satisfies conditions \eqref{PBOS}, and \eqref{intSPN}. Then, its optimal solution value $p^{\smash{(r)}}_{\MK}$ can be computed up to additive error $\varepsilon >0$ in time polynomial in $n$ and $\log(1/\varepsilon)$.
\end{thm}
\begin{proof}
    As shown above, we can formulate~\eqref{cop-program-approx1} as the semidefinite program~\eqref{eq:ref stqp1}. We observe that program~\eqref{eq:ref stqp1} has bit size polynomial in $n$ (recall that~\eqref{cop-program} has bit size polynomial in $n$ by assumption). To show~\eqref{eq:ref stqp1} can be solved (up to precision $\varepsilon$) in polynomial time in $n$ and $\log(1/\varepsilon)$, we use \Cref{thm:approx sol} and \Cref{coro-SoS}.
    
    Let $\mathcal{F}$ be the feasible set of program \eqref{eq:ref stqp1}, i.e.,
    \begin{equation*}
        \mathcal F = \Big \{
        \tilde Q = \left(\begin{array}{c|c} 
        Q & 0\\
        \hline 
        0 & D  
        \end{array}\right) \colon
        \langle A_\gamma,\tilde Q\rangle = b_\gamma,\ \tilde Q\succeq 0\Big \}.
    \end{equation*}
    Then $\mathcal F$ is contained in the set $\mathcal S^l$, for $l=k_{r+2}^n+2m$, where $k_{r+2}^n=\binom{n+r+1}{r+2}$ is the number of monomials of degree $r+2$ in $n$ variables.
    In order to apply Theorem \ref{thm:approx sol}, we need to find a matrix $X_0\in \mathcal{F}$ with rational entries that can be encoded in $\poly(n)$ bits, and rational numbers $R_1,R_2$ with bit size $\poly(n)$ such that
        $$B(X_0,R_1)\subseteq \mathcal F\subseteq B(X_0,R_2),$$
    where the balls are taken in the affine space 
        $$V = V(\mathcal F) := \{\Tilde Q\in \mathcal S^l : \langle A_\gamma,\Tilde Q\rangle = b_\gamma\}.$$
    By condition \eqref{intSPN}, there exists $\bar y\in\mathbb R^m$ such that
    \begin{equation*}
        \sum_{i=1}^m\bar y_i A_i - C = P+N,
    \end{equation*}
    for $P\succ 0$ with $\lambda_{\min}(P)= a > 2^{-\poly(n)}$ and $N$ with nonnegative entries. Note that the hypotheses on $P$ imply also that there exists $b\in \mathbb{Q}$ of bit size $\poly(n)$, with $a>b\geq 2^{-\poly(n)}$, such that $P_b:=P-bJ\succeq 0$, where $J$ is the all ones matrix.
    
    Consider the polynomial $p_{\bar y}=(x^{\circ 2})^T(\sum_{i=1}^m \bar y_iA_i -C)x^{\circ 2}$. We can write
    \begin{align*}
        p_{\bar y} &=(x^{\circ 2})^T(\sum_{i=1}^m \bar y_iA_i -C)x^{\circ 2} \\
        &=(x^{\circ 2})^T(P+N)x^{\circ 2}\\
        &= (x^{\circ 2})^T P_b x^{\circ 2} + (x^{\circ 2})^T (bJ +N) x^{\circ 2},
    \end{align*}
    where $P_b\succeq 0$ and $bJ+N$ has all entries greater or equal to $b$. The polynomial~$p_{\bar y}$ can thus be written as
    \begin{align*}
        p_{\bar y} = (x^{\circ 2})^T P_b x^{\circ 2} + \sum_{|\alpha|=2} \tilde b_{\alpha} x_{\alpha}^{2\alpha}.
    \end{align*}
    It follows that the degree-$(2r+4)$ polynomial $(\sum_i x_i^2)^r p_{\bar y}$ is an SoS of the form
    \begin{equation}\label{eq: SoS1}
        (\sum_{i=1}^n x_i^2)^r p_{\bar y} = (\sum_{i=1}^n x_i^2)^r (x^{\circ 2})^T P_b x^{\circ 2} + \sum_{\vert \alpha\vert = r+2} b_{\alpha}x_{\alpha}^{2\alpha},
    \end{equation}
    where $b_\alpha \geq b$ for all $\alpha$. Hence, we have 
        $$\sum_{|\alpha|=r+2} b_{\alpha} x_{\alpha}^{2\alpha} = [x]_{r+2}^TQ'[x]_{r+2},$$
    where $Q'$ is the diagonal matrix with entries $Q'_{\alpha,\alpha}=b_\alpha$. 

    We observe that 
        $$(\sum_{i=1}^n x_i^2)^r (x^{\circ 2})^T P_b x^{\circ 2}= [x]_{r+2}^T P'[x]_{r+2},$$
    where $P' \succeq 0$ is a positive semidefinite matrix with rational entries of size $\poly(n)$. Indeed, the entries of $P'$ are linear combinations of entries of $P_b$: we observe that, for any $\alpha\in \mathbb{N}^n$ with $|\alpha|=r$, we have $x^{2\alpha}(x^{\circ 2})^T P_b x^{\circ 2} = [x]_{r+2}^TP_\alpha'[x]_{r+2}$, where $P_\alpha'$ is obtained from $P_b$ by adding some zero rows and columns. Then, we see that $P' = \sum_{ |\alpha| = {r}} {n \choose 2\alpha} P_{2\alpha}'$.

    Now, we have
    \begin{align}\label{eq: py-SoS1}
        (\sum_i x_i^2)^r p_{\bar y} = [x]_{r+2}^T (P'+Q')[x]_{r+2}.
    \end{align}
    We set
    $$X_0 = \left(\begin{array}{c|c} 
        P'+Q' & 0\\
        \hline 
        0 & D
        \end{array}\right),$$
    where $D=\mathrm{Diag}(2R-\bar{y}_1, 2R+\bar{y}_1, \dots, 2R-\bar{y}_n, 2R+\bar{y}_n)$. Observe that $X_0\in \mathcal{F}$ as $P'+Q'$ is positive semidefinite and \Cref{eq: py-SoS1} holds.
    Let ${R_1 = \min\{b/ k_{r+2}^n,R\}}$, and consider the ball
    \begin{equation*}
        B(X_0,R_1)= \Big \{
         C'\in\mathcal S^l \colon C'= \left(\begin{array}{c|c} 
        Q & 0\\
        \hline 
        0 & D  
        \end{array}\right), 
        \langle A_\gamma,C'\rangle = b_\gamma, Q\succeq 0,\ \Vert X_0 - C' \Vert_F < R_1 \Big \}.
    \end{equation*}
    We show that $B(X_0,R_1)\subseteq \mathcal{F}$. Let $X$ be a matrix in $B(X_0,R_1)$, $X$ can then be seen as a perturbation of $X_0$ of the form
    $$X = \left(\begin{array}{c|c} 
        P'+Q' + \Tilde E & 0\\
        \hline 
        0 & D + \varepsilon
        \end{array}\right).$$
    where the entries of $\Tilde E$ and $\varepsilon=\mathrm{Diag}(\varepsilon_1,\dots,\varepsilon_{2m})$ are bounded in absolute value by~$R_1$. 
    Note that $Q'+\Tilde E$ is diagonally dominant, and hence it is positive semidefinite, meaning $P'+Q'+\Tilde E$ is positive semidefinite. Furthermore, we have $(D+\varepsilon)_{ii}>0$. We conclude that $X$ is positive semidefinite and hence $X \in \mathcal F$.

    Next, we have to show that there exists an $R_2\leq 2^{\poly(n)}$ such that $\mathcal F \subseteq B(X_0,R_2)$. For this, we show that for every matrix $\tilde Q \in \mathcal F$, the absolute values of the entries of $\tilde Q$ are bounded by $2^{\poly(n)}$. Let $\tilde{Q}\in \mathcal{F}$, i.e.,
    \begin{equation*}
        \tilde Q = \left(\begin{array}{c|c} 
        Q & 0\\
        \hline 
        0 & D  
        \end{array}\right) \text{ with }
        \langle A_\gamma,\tilde Q\rangle = b_\gamma,\ \tilde Q\succeq 0.
    \end{equation*}
    Since $Q$ is positive semidefinite, it has a spectral decomposition $Q=\sum_i \lambda_i u_i u_i^T$ with $\lambda_i \geq 0$ for all $i$. Writing $v_i=\sqrt{\lambda_i}u_i$, for all $i\in [n]$, we have
    $Q= \sum_i v_i v_i^T$.
    Then, it holds that 
        $$[x]_{r+2}^T Q [x]_{r+2}=\sum_i (v_i^T [x]_{r+2})^2= (\sum_{i=1}^nx_i^2)^r(x^{\circ 2})^T(\sum_{i=1}^m y_i A_i - C)x^{\circ 2}$$
    for some $y\in [-2R, 2R]^m$. Thus, $(\sum_{i=1}^m y_i A_i - C)\in \MK_n^{\smash{(r)}}$.  By Corollary \ref{coro-SoS}, the entries of $v_i$ are thus bounded by $2^{\poly(n)}$, and so the entries of $Q$ are bounded by $2^{\poly(n)}$. We conclude by noting that the entries of $D$ are all bounded by $2R$. 
\end{proof}

\subsection{Polynomial-time computation of \texorpdfstring{\eqref{eq:conic-prog2}}{CP-Q}}

Recall that a matrix ${M\in \MS^n}$ belongs to the cone $\MQ_n^{\smash{(r)}}$ if 
    $$(\sum_{i=1}^n x_i)^r x^T M x = \sum_{\substack{\beta\in \mathbb{N}^n,\\ |\beta|=r}} x^\beta\sigma_\beta + \sum_{\substack{\beta\in \mathbb{N}^n,\\ |\beta|=r+2}}c_\beta x^\beta,$$
where $\sigma_\beta$ is a homogeneous sum of squares of degree~$2$ for all $\beta \in \mathbb{N}^n$ with ${|\beta|=r}$, and $c_\beta \geq 0$ is a nonnegative scalar for each $\beta\in \mathbb N^n$ with $|\beta|=r+2$.
The relaxation~\eqref{eq:conic-prog2} is obtained by replacing the copositive cone~$\COP_n$ in~\eqref{cop-program}  by the cones $\MQ_n^{\smash{(r)}}$, i.e.,
\begin{align} \tag{\ref{eq:conic-prog2}} \label{cop-program-approx3}
    p^{\smash{(r)}}_{\MQ} := \min_{y\in \mathbb{R}^m}\Big\{ b^Ty: \sum_{i=1}^m y_i A_i-C \in \MQ_n^{\smash{(r)}}\Big\}.
\end{align}
In order to prove \Cref{THM:main} for this relaxation we again first formulate it as a semidefinite program. For $y\in \mathbb{R}^m$, consider the polynomial 
\[
    q_y(x) := x^T\Big(\sum_{i=1}^m y_iA_i - C\Big)x.
\]
Under conditions \eqref{PBOS} and \eqref{intSPN}, program (\ref{cop-program-approx3}) is equivalent to a semidefinite program of the form:
\begin{equation}\label{sdp-Q}
    \min \Big \{\Big\langle \tilde Q, 
    \left(\begin{array}{c|c|c} 
    0 & 0 & 0\\
    \hline 
    0 & 0 & 0 \\ 
    \hline
    0 & 0 & B
    \end{array}\right)\Big\rangle \colon
    \tilde Q = \left(\begin{array}{c|c|c} 
    Q & 0 & 0\\
    \hline 
    0 & C & 0 \\ 
    \hline
    0 & 0 & D
    \end{array}\right), 
    \langle \tilde Q, A_\gamma\rangle = b_\gamma, \ \tilde Q\succeq 0\Big \}.
\end{equation}
Here, the matrix $Q$ models a block-matrix of the form
\begin{align}\label{Q}
    {\footnotesize{
    Q= \setlength{\fboxsep}{0pt}
    \left(\begin{array}{@{}c@{}c@{}c@{\mkern-5mu}c@{\,}cc*{2}{@{\;}c}@{}}%
    \noalign{\vskip 1.5ex}
      \fbox{\ \ \,$Q_{(1,\dots,1)}$\ \ \,}
     \\[-0.1pt]
      &
       & \;\ddots \\
       & & \hskip-0.4pt & \hskip-0.4pt \fbox{\ \, $Q_{\beta}$ \ \,} \\
       & & & \hspace{0.8cm }\;\ddots \\[-0.5ex]
       & & & &   \fbox{\ \ \,$Q_{(n,\dots,n)}$\ \ \,}
    \end{array}\right),}}
\end{align}
where we have one block for each monomial $x^\beta$ of degree $r$ ($\beta\in \mathbb{N}^n, |\beta|=r)$. Each block is an $n\times n$ matrix that models the sum of squares $\sigma_\beta$. That is, the equality
    $$\sigma_\beta=[x]_1^TQ_\beta[x]_1, \quad Q_\beta \succeq 0.$$
The matrix $C\in \MS^{\binom{n+r+1}{r+2}}$ models a diagonal matrix with entries $C_{\beta \beta} = c_{\beta}$ for ${\beta\in \mathbb{N}^{n}}$ with $|\beta|=r+2$.
Finally, the matrix $D$ models a $2n\times 2n$ diagonal matrix encoding the variables $y_i$ in the following way. We impose linear constraints so that $D= \mathrm{Diag}(d_1^+,d_1^-,\dots,d_m^+,d_m^-)$, where, for each $i\in [m]$, the variable $d_i^+$ models $2R+y_i$, and we impose the linear constraint
    $$ d_i^- = 4R - d_{i}^+.$$
(Here, $R \leq 2^{\poly(n)}$ is as in \Cref{obs:R}.)
Hence, for each $i\in [m]$, the variable~$d_i^-$ models $2R-y_i$. Finally, the matrix $B\in \mathcal S^{2m}$ is the diagonal matrix \linebreak $\mathrm{Diag}(b_1/2,-b_1/2,\dots,b_m/2,-b_m/2)$.

The linear constraints of the SDP~\eqref{sdp-Q} are thus given by the following identity:
    $$(\sum_{i=1}^nx_i)^rx^Tq_yx = \sum_{\substack{\beta\in \mathbb{N}^n,\\ |\beta|=r}} x^\beta[x]_1^T Q_\beta [x]_1 + \sum_{\substack{\beta\in \mathbb{N}^n,\\ |\beta|=r+2}}c_\beta x^\beta.$$

\begin{thm}[Special case of \Cref{THM:main}]\label{thm:bounds-Q}
    Let $r\in \mathbb{N}$. Assume that \eqref{cop-program-approx3} satisfies conditions \eqref{PBOS} and \eqref{intSPN}. Then, its optimal solution $p^{\smash{(r)}}_{\MQ}$ can be computed up to additive error $\varepsilon > 0$ in time polynomial in $n$ and $\log(1/{\varepsilon})$.
\end{thm}
\begin{proof}
    As shown above, we can formulate \eqref{cop-program-approx3} as the semidefinite program \eqref{sdp-Q}. We observe that \eqref{sdp-Q} has polynomial bit size in $n$ (recall that \eqref{cop-program} has polynomial bit size in $n$ by assumption).
    To show that~\eqref{sdp-Q} can be solved (up to precision $\varepsilon$) in time polynomial in $n$ and $\log(1/\varepsilon)$, we apply Theorem~\ref{thm:approx sol} and \Cref{coro-SoS}. 
    
    Let $\mathcal{F}$ be the feasible set of program \eqref{sdp-Q}, i.e.,
    $$ \mathcal{F}=\left\{  \tilde Q = \left(\begin{array}{c|c|c} 
        Q & 0 & 0\\
        \hline 
        0 & C & 0 \\ 
        \hline
        0 & 0 & D
        \end{array}\right), 
        \langle \tilde Q, A_\gamma\rangle = b_\gamma, \ \tilde Q\succeq 0 \right\}. $$
    To apply Theorem~\ref{thm:approx sol}, we need to find a rational matrix $X_0\in \mathcal{F}$ and rational scalars $R_1,R_2$ that can be encoded in $\poly(n)$ bits such that 
        $$B(X_0, R_1)\subseteq \mathcal{F} \subseteq B(X_0,R_2),$$
    where the ball is taken in the affine subspace 
        $$V(\mathcal F):= \{\Tilde Q\in \mathcal{S}^{l}: \langle \tilde Q , A_\gamma\rangle = b_\gamma\}.$$
    where $l=n\binom{n+r-1}{r}+ \binom{n+r+1}{r+2}+2m$.
    Let $\bar y\in \mathbb{Q}^m$, $P$ and $N$ be as in condition~\eqref{intSPN}. That is, we have  
        $$\sum_{i=1}^m\bar y_i A_i-C=P+N,$$
    with $\lambda_{\min}(P)\geq 2^{-\poly(n)}$. 
    This implies that there exists a $b\geq 2^{-\poly(n)}$ with bit size at most $\poly(n)$, such that $P_b:=P-bJ\succeq 0$. We write 
        $$q_{\bar y}(x) = x^TP_bx + x^T(bJ+N)x.$$
    Observe that the second term on the RHS above is a polynomial whose coefficients are all at least $b$. Therefore, we have 
    \begin{align}\label{eq-P_b}
        (\sum_{i=1}^nx_i)^rq_{\bar y}(x)= \sum_{\substack{\beta\in \mathbb{N}^n \\ |\beta|=r}}x^\beta x^T (a_\beta P_b)x + \sum_{\substack{\beta\in \mathbb{N}^n \\ |\beta|=r+2}} c_\beta x^\beta,
    \end{align}
    where $a_\beta\geq 1$ is the coefficient of $x^\beta$ in $(\sum_{i=1}^nx_i)^r$, and $c_\beta\geq b$ for all $\beta$. 
    
    Notice that the coefficients of the polynomial 
        $$\sum_{\substack{\beta\in \mathbb{N}^n \\ |\beta|=r}}x^\beta (x^TIx)=(\sum_{\substack{\beta\in \mathbb{N}^n \\ |\beta|=r}}x^\beta) (\sum_{i=1}^nx_i^2)$$
    are upper bounded by $n$. This allows us to rewrite~\eqref{eq-P_b} as 
        $$(\sum_{i=1}^nx_i)^rq_{\bar y}(x)= \sum_{\substack{\beta\in \mathbb{N}^n \\ |\beta|=r}}x^\beta x^T (a_\beta P_b + \tfrac{b}{2n}I)x + \sum_{\substack{\beta\in \mathbb{N}^n \\ |\beta|=r+2}} \tilde{c_\beta} x^\beta,$$
    where $\tilde{c}_\beta\geq b/2$ for all $\beta\in \mathbb{N}^n$, with $|\beta|=r+2$. This expression gives us a feasible point $X_0$ of~\eqref{sdp-Q} as follows:
        $$X_0= \left(\begin{array}{c|c|c} 
            Q' & 0 & 0\\
            \hline 
            0 & C' & 0 \\ 
            \hline
            0 & 0 & D'
            \end{array}\right), $$
    where, 
    \begin{align}
        {\footnotesize{
        Q'= \setlength{\fboxsep}{0pt}
        \left(\begin{array}{@{}c@{}c@{}c@{\mkern-5mu}c@{\,}cc*{2}{@{\;}c}@{}}%
        \noalign{\vskip 1.5ex}
          \fbox{\ \ \,$a_{(1,\dots,1)} P_b + \tfrac{b}{2n}I$\ \ \,}
         \\[-0.1pt]
          &
           & \;\ddots \\
           & & \hskip-0.4pt & \hskip-0.4pt \fbox{\ \, $a_\beta P_b + \tfrac{b}{2n}I$ \ \,} \\
           & & & \hspace{0.8cm }\;\ddots \\[-0.5ex]
           & & & &   \fbox{\ \ \,$a_{(n,\dots,n)} P_b + \tfrac{b}{2n}I$\ \ \,}
        \end{array}\right),}}
    \end{align}
    $C'\in \MS^{\binom{n+r+1}{r+2}}$ is a diagonal matrix with $C'_{\beta \beta}= \tilde{c}_{\beta}$, and $D'$ is also a diagonal matrix defined as $D':=\text{Diag}(2R+\bar{y}_1, 2R-\bar{y}_1, \dots, 2R+\bar{y}_m, 2R-\bar{y}_m)$.
    
    Now, we set $R_1:=\min \{b/4n^2,R\}$, and we show that
        $$B(X_0, R_1)\subseteq \mathcal{F}.$$
    
    Let $X\in B(X_0, R_1)$. We shall show that $X$ is positive semidefinite. The upper block of $X$ is of the form
    {\footnotesize{
    \begin{align*}
         \setlength{\fboxsep}{0pt}
        \left(\begin{array}{@{}c@{}c@{}c@{\mkern-5mu}c@{\,}cc*{2}{@{\;}c}@{}}%
        \noalign{\vskip 1.5ex}
          \fbox{\ \ \,$a_{(1,\dots,1)} P_b + \tfrac{b}{2n}I + E_{(1, \dots, 1)}$\ \ \,}
         \\[-0.1pt]
          &
           & \;\ddots \\
           & & \hskip-0.4pt & \hskip-0.4pt \fbox{\ \, $a_\beta P_b + \tfrac{b}{2n}I + E_\beta$ \ \,} \\
           & & & \hspace{0.8cm }\;\ddots \\[-0.5ex]
           & & & &   \fbox{\ \ \,$a_{(n,\dots,n)} P_b + \tfrac{b}{2n}I + E_{(n,\dots,n)}$\ \ \,}
        \end{array}\right),
    \end{align*}}}
    where the absolute value of all entries of the matices $E_{\beta}$ are bounded by $R_1$. Observe that, for all $\beta$, we have $a_\beta P_b + (b/2n)I + E_{\beta} \succeq 0$, because $(b/2n)I + E_{\beta}$ is a diagonally dominant matrix. Similarly,  the second block of the matrix $X$ is of the form 
        $$\mathrm{Diag}(\tilde{c}_\beta) + \mathrm{Diag}(\varepsilon_{(1,\dots,1)}, \dots, \varepsilon_{\beta}, \dots, \varepsilon_{(n,\dots,n)}),$$
    where $\varepsilon_{\beta}\leq R_1$. Since $\tilde{c}_\beta\geq b/2$, this diagonal matrix is positive semidefinite. Finally, the third block of $X$ is of the form
        $$\mathrm{Diag}(2R +\bar y_1, 2R-\bar y_1, \dots, 2R + \bar y_m, 2R-\bar y_m)+ \mathrm{Diag}(\varepsilon_{1}^+, \varepsilon_{1}^-\dots, \varepsilon_{m}^+, \varepsilon_{m}^-),$$
    where $\varepsilon_i^{+}, \varepsilon_{i}^-\leq R_1\leq R$, and thus this diagonal matrix is also positive semidefinite. Therefore $B(X_0, R_1)\subseteq \mathcal{F}$ as desired. 
    
    Finally, we show that there exists an $R_2 \leq 2^{\poly(n)}$ such that $ \mathcal{F} \subseteq B(X_0, R_2)$.
     Let $\tilde{Q}\in \mathcal{F}$, i.e.,
        \begin{equation*}
            \tilde Q = \left(\begin{array}{c|c|c} 
                Q & 0 & 0\\
                \hline 
                0 & C & 0 \\ 
                \hline
                0 & 0 & D
                \end{array}\right), 
                \langle \tilde Q, A_\gamma\rangle = b_\gamma, \ \tilde Q\succeq 0,
        \end{equation*}
    where, $Q$, $C$ and $D$ take the shape as described below formulation (\ref{sdp-Q}), i.e., $Q$ takes the form of the matrix as in (\ref{Q}), $C\in \MS^{\binom{n+r+1}{r+2}}$ is a diagonal matrix with $C_{\beta \beta}=c_\beta$ for some $c_\beta\geq 0$, and $D=\mathrm{Diag}(2R+\tilde{y}_1, 2R-\tilde{y}_1, \dots, 2R+\tilde{y}_m, 2R-\tilde{y}_m)$ for some $\tilde{y}\in [-2R,2R]^m$. Moreover, the following identity holds:
        \begin{align*}
            (\sum_{i=1}^nx_i)^rx^T(\sum_{i=1}^m \tilde{y}_iA_i -C)x =& \sum_{\substack{\beta\in \mathbb{N}^n,\\ |\beta|=r}} x^\beta[x]_1^T Q_\beta [x]_1 + \sum_{\substack{\beta\in \mathbb{N}^n,\\ |\beta|=r+2}}c_\beta x^\beta.
         \end{align*}   
    Since $Q_\beta\succeq 0$ for all $\beta\in \mathbb{N}^n$ with $|\beta|=r$, we can write 
        $$Q_\beta = \sum_{j}v_{\beta_j}v_{\beta_j}^T.$$
    Then we have 
        \begin{align*}
            (\sum_{i=1}^nx_i)^rx^T(\sum_{i=1}^m \tilde{y}_iA_i -C)x =&\sum_{\substack{\beta\in \mathbb{N}^n,\\ |\beta|=r}} x^\beta(\sum_{j=1}^n v_{\beta_j}[x]_1)^2 + \sum_{\substack{\beta\in \mathbb{N}^n,\\ |\beta|=r+2}}c_\beta x^\beta,
        \end{align*}
    So, $\sum_{i=1}^n\tilde{y}_iA_i -C\in \MQ_n^{\smash{(r)}}$ and all its entries are bounded by $2^{\poly(n)}$. By Corollary~\ref{coro-SoS}, we obtain the entries of $v_\beta$ (for $|\beta|=r)$ and $c_\beta$ (for $|\beta|=r+2$) are bounded by $2^{\poly(n)}$. This implies that the entries of $Q$, $C$ and $D$ are bounded by $2^{\poly(n)}$.   
\end{proof}

\section{Application to standard quadratic programs}\label{section:SQP}

Standard quadratic programs (SQPs) can be formulated as copositive programs. Thus, they can be approximated using the techniques of \Cref{SEC:relaxations}.
In this section, we recall the copositive formulation of a generic SQP, and we show that the corresponding upper bounds \eqref{eq:conic-prog1} and \eqref{eq:conic-prog2} can be computed in polynomial time (up to a fixed additive error). Moreover, it turns out that the \emph{reciprocal} of a standard quadratic program can be formulated as a copositive program, too (whenever it is well-defined). Under the condition that the matrix $M$ defining the SQP lies in the cone $\mathrm{SPN}_n^{{\rm low}}$ (see~\eqref{intSPN}), we show that the bounds \eqref{eq:conic-prog1} and \eqref{eq:conic-prog2} corresponding to the copositive formulation of the reciprocal SQP can be computed in polynomial time as well (up to a fixed additive error). 

Let $M\in \MS^n$ be a symmetric matrix. A \emph{standard quadratic optimization problem} is defined as the following minimization problem over the standard simplex $\Delta_n$: 
\begin{align}\tag{SQP}\label{SQ1}
    p_{\min} =\min \Big\{x^TMx: x\in \Delta_n\Big\}.
\end{align}
As shown in \cite{Bomze2000OnCP, dKP02}, we can formulate \eqref{SQ1} as a copositive program
\begin{align}\tag{CP-SQP}\label{COP-SQ1}
    p_{\min} = \max_{\lambda\in \mathbb{R}} \Big\{\lambda: M-\lambda J \in \COP_n\Big\},
\end{align}
or equivalently,
\begin{equation}\label{COP-SQ2}
    p_{\min} = -\min_{\lambda\in \mathbb{R}} \Big\{\lambda: M+\lambda J \in \COP_n\Big\}.
\end{equation}
Assume that $p_{\min }>0$. Then, we can also write $1/p_{\min}$ as the optimum of a copositive program (see Appendix \ref{appendix:cop-inv}). Namely, we have 
\begin{align}\tag{CP-INV}\label{COP-INVERSE1}
    \frac{1}{p_{\min}}=\min_{\lambda\in \mathbb{R}}\Big\{\lambda \colon \lambda M-J\in \COP_n\Big\}.
\end{align}

\subsection{Approximating SQPs}
The relaxations~\eqref{eq:conic-prog1} and~\eqref{eq:conic-prog1} of the copositive formulation~\eqref{COP-SQ1} of an SQP take the form:
\begin{equation}\label{COP-SQ-K1}
    p_\MK^{\smash{(r)}} = -\min_{\lambda \in \R} \Big\{\lambda: M+\lambda J \in \MK_n^{\smash{(r)}}\Big\} ~\geq p_{\min},
\end{equation}
\begin{equation}\label{COP-SQ-Q1}
    p_\MQ^{\smash{(r)}} = -\min_{\lambda \in \R} \Big\{\lambda: M+\lambda J \in \MQ_n^{\smash{(r)}}\Big\} ~\geq p_{\min}.
\end{equation}
As we show now, these bounds can always be computed in polynomial time (up to a given precision).
\begin{thm}[Restatement of \Cref{PROP:main_SQP}] \label{THM:SQP}
    Let $r\in \mathbb{N}$, and let $M\in \mathcal S^n$ with $\size(M) \leq \poly(n)$.  Then, the bounds~\eqref{COP-SQ-K1} and~\eqref{COP-SQ-Q1} on the minimum $p_{\min}$ of the SQP defined by $M$ can be computed up to additive error $\ep > 0$ in time polynomial in  $n$ and $\log(1/\ep)$.
\end{thm}
\begin{proof}
    We prove the result just for the bound~\eqref{COP-SQ-K1}. The proof for~\eqref{COP-SQ-Q1} is analogous. It suffices to show that \eqref{COP-SQ-K1} satisfies the conditions of \Cref{THM:main}. 
    
    First, we show that~\eqref{COP-SQ-K1} satisfies \eqref{intSPN}. For this, we set $\bar \lambda := \max \vert M_{ij}\vert +1 \leq 2^{\poly(n)}$. Then, the entries of the matrix $M+\bar\lambda J$ are all at least 1. Thus, we can write
        $$M+\bar\lambda J= \underbrace{I}_{P} + \underbrace{M+\bar\lambda J - I}_{N},$$
    showing that \eqref{COP-SQ-K1} satisfies \eqref{intSPN} (with interior point $\bar \lambda$).
    
    It remains to show that~\eqref{COP-SQ-K1} satisfies \eqref{PBOS}. Observe that $-\bar\lambda$ is not feasible for~\eqref{COP-SQ-K1}, as all entries of $M-\bar \lambda J$ are negative. Therefore, since the cones $\MK_n^{\smash{(r)}}$ are closed, \eqref{COP-SQ-K1} must have an optimal solution in the interval $[-\bar\lambda, \bar\lambda]$.
\end{proof}

\subsection{Approximating reciprocal SQPs}
Now we assume that $p_{\min}>0$, and we consider the copositive program~\eqref{COP-INVERSE1}. The relaxations~\eqref{eq:conic-prog1} and~\eqref{eq:conic-prog2} of~\eqref{COP-INVERSE1} take the form:
\begin{align}
    q_\MK^{\smash{(r)}} &= \min \{\lambda \colon \lambda M-J\in \mathcal K_n^{\smash{(r)}} \} ~\geq \frac{1}{p_{\min}},\label{eq: sdp-knr1}\\
    q_\MQ^{\smash{(r)}} &= \min \{\lambda \colon \lambda M-J\in \mathcal Q_n^{\smash{(r)}} \} ~\geq \frac{1}{p_{\min}}.\label{eq: sdp-qnr1}
\end{align}
We show that we can compute these bounds in polynomial time (up to a small additive error) under the assumption that $M$ belongs to the set $\SPN^{\rm low}_n$ (see \eqref{intSPN}).

\begin{thm} \label{THM:COPINV}
Let $r\in \mathbb{N}$, and let $M \in \SPN^{\rm low}_n$ (thus with $\size(M) \leq \poly(n)$).  Then, the bounds~\eqref{eq: sdp-knr1} and~\eqref{eq: sdp-qnr1} on the reciprocal $1/p_{\min}$ of the minimum of the SQP defined by $M$ can be computed up to additive error $\ep > 0$ in time polynomial in  $n$ and $\log(1/\ep)$.
\end{thm}
\begin{proof}
    We prove the result just for the bound~\eqref{eq: sdp-knr1}. The proof for~\eqref{eq: sdp-qnr1} is analogous. It suffices to show that \eqref{eq: sdp-knr1} satisfies the conditions of \Cref{THM:main}.
        
    First, we show that~\eqref{eq: sdp-knr1} satisfies \eqref{intSPN}. Since $M\in \SPN^{\rm low}_n$, we have that
        $$M=P+N,$$
    where $P$ and $N$ are matrices that can be encoded in $\poly(n)$ bits, with $P\succ 0$, $N\geq 0$, and $a:=\lambda_{\min}(P) \geq 2^{-\poly(n)}$. Let $b\in \mathbb{Q}$ with $\size(b) \leq \poly(n)$ such that $2^{-\poly(n)} \leq b<a$. We set $\bar \lambda=4n/b$, so that $\bar\lambda$ can be encoded in $\poly(n)$ bits. Then, $(4n/b)M-J = (4n/b)(P+N)-J = (4n/b)P-J + (4n/b)N$. Using Weyl's inequality, we find that $\lambda_{\min}((4n/b)P-J)\geq n$, thus showing that ~\eqref{eq: sdp-knr1} satisfies \eqref{intSPN} (with interior point $\bar\lambda$).
    
    Next, we show that \eqref{eq: sdp-knr1} satisfies \eqref{PBOS}. We have seen that $\bar \lambda$ is feasible for~\eqref{eq: sdp-knr1}. On the other hand, it is straightforward to check that $-\bar\lambda$ is not feasible. Since the cones $\MK_n^{\smash{(r)}}$ are closed, \eqref{eq: sdp-knr1} must have an optimal solution in the interval~$[-\bar\lambda, \bar\lambda]$. 
\end{proof}

\begin{obs} \label{obs:rel q-p}
    It turns out (see \Cref{appendix:rel q-p}) that the relaxations \eqref{COP-SQ-K1} and \eqref{eq: sdp-knr1} to an SQP and its reciprocal, respectively, satisfy
    \begin{equation}\label{eq:rel q-p}
        q_\MK^{\smash{(r)}}= \frac{1}{p_{\MK}^{\smash{(r)}}}.
    \end{equation}
    For this reason, it may seem unnecessary to apply \Cref{THM:COPINV} to the relaxation~\eqref{eq: sdp-knr1} of the reciprocal SQP, rather than applying \Cref{THM:SQP} to the relaxation~\eqref{COP-SQ-K1} of the original SQP and taking the reciprocal. However, we note that these theorems only guarantee that $q_\MK^{\smash{(r)}}$ and $p_{\MK}^{\smash{(r)}}$ can be computed up to an additive error~$\ep > 0$. Thus, taking the reciprocal of an approximately optimal solution to~\eqref{COP-SQ-K1} yields an approximation
    \[
        \frac{1}{p_{\MK}^{\smash{(r)}} \pm \ep} \approx \frac{1}{p_{\MK}^{\smash{(r)}}} = q_\MK^{\smash{(r)}}
    \]
    of $q_\MK^{\smash{(r)}}$ which is potentially far worse than the approximation $q_\MK^{\smash{(r)}} \pm \ep$ guaranteed by applying~\Cref{THM:COPINV} directly to~\eqref{eq: sdp-knr1}.  
    An analogous statement applies to the relaxations~\eqref{COP-SQ-Q1} and \eqref{eq: sdp-qnr1}. 
\end{obs}

\subsection{Weighted stability number of a graph}
Let  $G=(V=[n],E)$, and let $\omega\in \mathbb{R}^{V}$ be a vector of positive vertex weights, i.e.,  $\omega_i>0$ for all $i\in V$. The {\em weighted stability number} $\alpha(G,\omega)$ denotes the maximum weight $\omega(S)=\sum_{i\in S}\omega_i$ of a stable set $S$. The following formulation of $1/\alpha(G,\omega)$ was shown in \cite{Gibbons1996} as a generalization of the Motzkin-Straus formulation \eqref{EQ:alpha_SQP} shown in the introduction. Consider the following set of matrices: 
\begin{align}\label{ms-w}
\begin{split}
    \mathcal{M}(G,\omega)=\{B\in \mathcal{S}^n:B_{ii}=\frac{1}{\omega_i} \text{ for $i\in V$}, B_{ij}&\geq \frac{1}{2}(B_{ii}+B_{ij}) \text{ for } \{i,j\}\in E, \\ 
    B_{ij}&=0 \text{ for } \{i,j\}\in \overline E\}.
\end{split}    
\end{align}
Then, for every $B\in \mathcal{M}(G,\omega)$, it holds
\begin{align}
    \frac{1}{\alpha(G,\omega)}= \min\{x^TBx \colon x\in \Delta_n\},
\end{align}
and we have the following copositive formulation for the weighted stability number of the graph $G$, \cite{dKP02}:
\begin{align}
    \alpha(G,\omega)=\max_{t\in \mathbb{R}}\{t:tB-J\in \COP_n\}.
\end{align}
The following relaxations $\vartheta^{\smash{(r)}}(G,\omega)$ and $\nu^{\smash{(r)}}(G,\omega)$ were introduced, respectively, by de Klerk \& Pasechnik \cite{dKP02} and Peña, Vera \& Zuluaga \cite{PVZ}. We remark that these bounds were defined in the context of unweighted graphs, but they can be easily generalized (see \cite{LV22}). The relaxations are respectively given by
\begin{align}
    \vartheta^{\smash{(r)}}(G,\omega)=\max_{t\in \mathbb{R}}\{t:tB-J\in \MK_n^{\smash{(r)}}\}, \\
    \nu^{\smash{(r)}}(G,\omega)=\max_{t\in \mathbb{R}}\{t:tB-J\in \MQ_n^{\smash{(r)}}\}.
\end{align}
Then, we have that $\vartheta^{\smash{(r)}}(G,\omega)\to \alpha(G,\omega)$ and $\nu^{\smash{(r)}}(G,\omega)\to \alpha(G,\omega)$ as $r\to \infty$. For the unweighted case, i.e., $\omega=(1, \dots, 1)$, de Klerk \& Pasechnik \cite{dKP02} showed that the bounds $\vartheta^{\smash{(0)}}(G)$ (and $\nu^{\smash{(0)}}(G)$) are at least as good as the well-known bound $\vartheta(G)$ introduced by Lovász in his seminal paper \cite{Lov}. We refer to \cite{GL07}, \cite{LV23}, \cite{SV} for further results about the exactness of these hierarchies. 
 
We show that $\vartheta^{\smash{(r)}}(G,\omega)$ and $\nu^{\smash{(r)}}(G,\omega)$ can be computed in polynomial time (up to a fixed precision).
\begin{cor}
Let $G=(V=[n],E)$ be a graph, let $\omega\in \mathbb{R}_+^V$ be a vector of weights, and let $B\in \mathcal{M}(G,\omega)$. Assume that $\size(\omega)=\poly(n)$ and $\size(B)=\poly(n)$. Then, the parameters $\vartheta^{\smash{(r)}}(G,\omega)$ and $ \nu^{\smash{(r)}}(G,\omega)$ can be computed up to an additive error $\varepsilon>0$ in time polynomial in $n$ and $\log(1/\ep)$.  
\end{cor}
\begin{proof}
    We apply Theorem \ref{THM:COPINV} with $M=B$. Observe that $B\in \SPN_n^{\text{low}}$ because it can be decomposed as $B=\text{Diag}(1/\omega_1, \dots, 1/\omega_n) + B'$, where $B'\geq 0$.   
\end{proof}

\section{Proof of \Cref{THM:main2} and an application}\label{section:bounded-CP}

In \Cref{section:main-results}, we have seen that the relaxations \eqref{cop-program-approx1} and \eqref{cop-program-approx3} of a copositive program~\eqref{cop-program} can be computed in polynomial time (up to a small additive error) if  they satisfy both conditions \eqref{PBOS} and \eqref{intSPN}. Recall that, by definition, condition~\eqref{intSPN} is satisfied by the relaxations if and only if it is satisfied by the original copositive program~\eqref{cop-program}. The same is not true for~\eqref{PBOS}. Indeed, as we will see in \Cref{section:examples}, there are examples where the program \eqref{cop-program} has a polynomially bounded optimal solution, but its relaxations do not.

Nonetheless, we show in this section that, as long as the {\em original} copositive program \eqref{cop-program} satisfies properties \eqref{PBOS} and  \eqref{intSPN}, then we can construct a \emph{new} copositive program with the same optimal value whose \emph{relaxations} satisfy~\eqref{PBOS} and \eqref{intSPN}. In particular, these relaxations can be computed in polynomial time (up to a small additive error) using \Cref{THM:main}.

\subsection{A modified copositive program}\label{new-bounds}
Here we assume that \eqref{cop-program} has a polynomially bounded optimal solution $y^*$, and that it satisfies \eqref{intSPN} with vector $\bar y \in \R^m$. Let $R \leq 2^{\poly(n)}$ be such that $y^*, \bar y \in [-R,R]^m$ (see \Cref{obs:R}) and let $D=\mathrm{Diag}(2R-y_1,2R+y_1,\dots,2R-y_m,2R+y_m) \in \MS^{2m}$. We consider the program
\begin{align}\label{CP-1}\tag{CP'}
    p' = \min\Big\{ b^Ty: M_y= \left(\begin{array}{c|c} 
        \sum_{i=1}^my_iA_i-C & 0\\
        \hline 
        0 & D
        \end{array}\right) \in \COP_{n+2m}\Big\}.
\end{align}
We can write $M_y=\sum_{i=1}^m y_i \bar A_i - \bar C$, with $\bar A_i, \bar C \in \mathcal S^{n+2m}$ defined as
\begin{align}\label{A-bar}
{\footnotesize{
 \bar A_i= \left(\begin{array}{c|c} 
            A_i & 0\\
            \hline 
            0 &  \begin{smallmatrix}
              \ddots & & & &\\
               & -1 & & &\\
               & & 1  & &\\
               & & & & \ddots
            \end{smallmatrix}
            \end{array}\right), \quad \bar C= \left(\begin{array}{c|c} 
            C & 0\\
            \hline 
            0 &  \begin{smallmatrix}
            & & \\
            & & \\
              - 2R & & \\
               & \ddots & \\
               & & \\
               & & -2R \\
               & &
            \end{smallmatrix}
            \end{array}\right).    
            }}
\end{align}
Thus, program (\ref{CP-1}) is a copositive program of the form (\ref{cop-program}).
We observe that $p^*=p'$, i.e.,  the optimum of problems (\ref{cop-program}) and (\ref{CP-1}) are equal. Indeed, we have that $p^*\leq p'$ because every feasible point for problem~(\ref{CP-1}) is also feasible for problem~(\ref{cop-program}), as the cone $\COP$ is closed under taking principal submatrices. Additionally, the optimal solution $y^*$ of problem \eqref{cop-program} is also feasible for program~\eqref{CP-1}. 

The relaxations \eqref{eq:conic-prog1}, \eqref{eq:conic-prog2} of~\eqref{CP-1} take the following form:
\begin{align}
    p{'}_{\MK}^{\smash{(r)}} = \min\Big\{ b^Ty: M_y= \left(\begin{array}{c|c} 
        \sum_{i=1}^my_iA_i-C & 0\\
        \hline 
        0 &  D 
        \end{array}\right) \in \MK_{n+2m}^{\smash{(r)}}\Big\} \label{CP-1-Knr},  \tag{CP'-K}\\
    p{'}_{\MQ}^{\smash{(r)}} = \min\Big\{ b^Ty: M_y= \left(\begin{array}{c|c} 
        \sum_{i=1}^my_iA_i-C & 0\\
        \hline 
        0 &  D 
        \end{array}\right) \in \MQ_{n+2m}^{\smash{(r)}}\Big\}\label{CP-1-Qnr} \tag{CP'-K}.
\end{align}

We have the following result, which immediately implies~\Cref{THM:main2}.
\begin{thm}\label{thm-new-bound}
    Let $r\in \mathbb{N}$. If \eqref{cop-program} satisfies conditions~\eqref{PBOS} and~\eqref{intSPN}, then the optimal values of the relaxations \eqref{CP-1-Knr} and \eqref{CP-1-Qnr} of the equivalent copositive program~\eqref{CP-1} can be computed up to additive error $\varepsilon > 0$ in time polynomial in $n$ and $\log(1/\varepsilon)$.
\end{thm}
\begin{proof}
     It suffices to show that~\eqref{CP-1-Knr} and~\eqref{CP-1-Qnr} satisfy the conditions of \Cref{THM:main}, namely~\eqref{PBOS} and~\eqref{intSPN}. We will do so only for~\eqref{CP-1-Knr}, as the proof for~\eqref{CP-1-Qnr} is analogous. 
     
     First, we observe that program~\eqref{CP-1-Knr} satisfies~\eqref{PBOS}. Indeed, by construction, every feasible solution of program~\eqref{CP-1-Knr} lies in the box $[- 2R,  2R]^m$. Furthermore, the optimum is attained because the cones $\MK^{\smash{(r)}}_{n + 2m}$ are closed. It remains to show that program~\eqref{CP-1-Knr} satisfies~\eqref{intSPN}. 
     Since the original program~\eqref{cop-program} satisfies~\eqref{intSPN}, there is a $\bar y \in \R^m$ with
        $$\sum_{i=1}^m \bar y_i A_i - C = P+N$$
    for some matrices $P,N\in \mathcal{S}^n$ of bit size $\leq \poly(n)$ such that $\lambda_{\min}(P)\geq 2^{-\poly(n)}$, and $N\geq 0$. Then, using~\eqref{A-bar}, the matrix $M_{\bar y}$ can be written as
        $$M_{\bar y}= \underbrace{\left(\begin{array}{c|c} 
            P & 0\\
            \hline 
            0 &  \begin{smallmatrix}
               &&&&\\
               2R-\bar y_1 & & & &\\
               & 2R+\bar y_1 & & &\\
               & & \ddots & &\\
               & & &  2R-\bar y_m &\\
               & & & &  2R+\bar y_m
            \end{smallmatrix}
            \end{array}\right)}_{P'}+
            \underbrace{
            \left(\begin{array}{c|c} 
            N & 0\\
            \hline 
            0 &  0
            \end{array}\right)
            }_{N'}
            .$$
    In addition, the eigenvalues of $P'$ are given by the eigenvalues of $P$ and the numbers $2R-\bar y_1, 2R+\bar y_1,\dots, 2R-\bar y_m$ and $2R+\bar y_m$. As we have $\lambda_{\min}(P) \geq 2^{-\poly(n)}$ and $ 2R-\bar y_i, 2R+\bar y_i \geq 2^{-\poly(n)}$, we therefore have $\lambda_{\min}(P') \geq 2^{-\poly(n)}$.
    Since the vector $\bar{y}$, and the matrices $P'$ and $N'$ have bit size $\leq \poly(n)$, we obtain that program \eqref{CP-1-Knr} satisfies (\ref{intSPN}).
\end{proof}

\subsection{Chromatic number of a graph}
The {\em chromatic number} of a graph $G=(V,E)$, denoted $\chi(G)$, is the minimum number of colors required to color the vertices of $G$ such that no two adjacent vertices have the same color. The chromatic number can be formulated as a copositive program~\cite{GvozdenovicLaurent}. Namely, it equals
\begin{equation}\label{chromatic-number-cop}
    \max_{(y,z)\in \mathbb{R}^2} \Big \{ y \colon M_{y,z,t}=\frac{1}{n^2}(t-y)J + z(n(I + A_{G_t})-J) \in \COP_n \text{ for } 1\leq t \leq n \Big \},
\end{equation}
where $A_{G_t}$ is the adjacency matrix of the graph $G_t$, obtained as the cartesian product of the complete graph $K_t$ with $t$ vertices and $G$.\footnote{For any $t\in \mathbb N$, the node set of $G_t$ is given by $V(K_t)\times V(G) = \bigcup_{p=1}^t V_p$, where $V_p := \{ pi \colon i\in V(G)\}$ and $(pi,qj)\in E(G_t)$ if $p\neq q \text{ and }i=j$ or if $p=q \text{ and } ij\in E(G)$} We note that the stability number of this graph satisfies
\begin{align}\label{property:alpha}
    \alpha(G_t)= \begin{cases}
       n\quad \quad \quad \quad \text{ if } k\geq \chi(G) \\
       \leq n-1 \quad \ \text{ if } 1\leq k\leq \chi(G)-1. 
    \end{cases}
\end{align}

It is not clear whether \Cref{THM:main} applies to program~\eqref{chromatic-number-cop}. In particular, while~\eqref{intSPN} is satisfied, it is not clear whether the relaxations~\eqref{eq:conic-prog1} and~\eqref{eq:conic-prog2} have polynomially bounded optimal solutions. However, as we show now, the original copositive formulation is guaranteed to satisfy~\eqref{PBOS} and thus the approach described in \Cref{new-bounds} can be applied (see also \Cref{THM:main2}).

\begin{thm}
    The copositive formulation~\eqref{chromatic-number-cop} of the chromatic number of a graph satisfies conditions~\eqref{PBOS} and~\eqref{intSPN}. 
\end{thm}
\begin{proof}
  First, we show that program (\ref{chromatic-number-cop}) satisfies (\ref{PBOS}). For this, we show that $(y,z)=(\chi(G),1)$ is a feasible (and thus optimal) solution. We show that the matrix $M_{\chi(G), 1, t}$ is copositive for $t=1, \dots, n$. First, assume that $t\geq \chi(G)$. Then, we have that $\alpha(G_t)=n$, and thus the matrix $n(I+A_{G_t})-J$ is copositive (recall formulation (\ref{alpha})). Moreover, since $t-\chi(G)\geq 0$, the matrix $(1/n^2)(t-\chi(G))J$ is copositive. Then, we obtain that $M_{\chi(G), 1, t}$ is copositive. 
    Now, assume $t\leq \chi(G)-1$. By property \eqref{property:alpha}, this implies that $\alpha(G_t)\leq n-1$ and $M_{\chi(G), 1, t}$ is copositive being the sum of two copositive matrices:
    \begin{align*}
        M_{\chi(G), 1, t}=\underbrace{\frac{1}{n^2}(t-\chi(G))J+(I+A_{G_t})}_{M_1}+\underbrace{(n-1)(I+A_{G_t})-J}_{M_2}
    \end{align*}
    where $M_1=(I+A_{G_t})-(1/n^2)(\chi(G)-t)J$ is copositive since $n^2/(\chi(G)-t)\geq n-1$, and $M_2$ is copositive since $\alpha(G_t)\leq n-1$ (recall formulation (\ref{alpha})).
    
    We show that property \eqref{intSPN} is satisfied.
    Let $(\bar y, \bar z)=(-n^2,1)$. Then we have
    \begin{align*}
        M_{\bar y,\bar z,t}&=\frac{1}{n^2}(t+n^2)J + n(I + A_{G_t})-J \\
        &=\underbrace{\frac{1}{n^2}tJ + nA_{G_t} + (n-1)I}_N + \underbrace{I}_P
    \end{align*}
    where both $N \geq 0$ and $P$ have entries of bit size $\leq \poly(n)$, and $\lambda_{\min}(P)=1$.
\end{proof}

\section{Examples of pathological copositive programs}\label{section:examples}

In this section we give several  examples of pathological copositive programs that illustrate the necessity of the assumptions we make in our main results.

\begin{ex}\label{ex:khachiyan-type}
For $y \in \mathbb R^{n}$, let $M_y \in \MS^{2n-1}$ be the block diagonal matrix given by
\begin{equation*}
    M_y:= \begin{pmatrix}
              y_1  & \ -y_2\\
              -y_2 & 1
          \end{pmatrix} \oplus
          \dots \oplus
          \begin{pmatrix}
              y_{n-1}  & \ -y_n\\
              -y_n & 1
          \end{pmatrix} \oplus
          \begin{pmatrix}
              y_n -2
          \end{pmatrix}.
\end{equation*}
Consider the following copositive program:
    \begin{align}\label{ex:problem1}
        p=\min\Big\{ y_n: M_y \in \COP_{2n - 1}\Big\}.
    \end{align}    
By looking at the diagonal entries, we get that every feasible solution should satisfy $y_i\geq 0$ for $i\in[n-1]$ and $y_n\geq 2$. Observe that $y_i=2^{2^{n-i}}$ is feasible because each block is a positive semidefinite matrix. Thus, the optimal value of problem (\ref{ex:problem1}) is 2. Now, we show that every optimal solution $y$, we have $y_1\geq 2^{2^{n-1}}$. Since every $2\times 2$ diagonal block of $M_y$ is copositive, it must hold
    $$ \begin{matrix}
         (1 & y_{i+1}) \\
          & 
     \end{matrix}\left(\begin{matrix}
         y_i & -y_{i+1} \\
         -y_{i+1} & 1
     \end{matrix}\right)\left(\begin{matrix}
          1\\
          y_{i+1}
     \end{matrix}\right) = y_i-2y_{i+1}^2 + y_{i+1}^2\geq 0.$$
 Thus, $y_i \geq y_{i+1}^2.$  Using these inequalities recursively and the fact that $y_n\geq 2$,  we obtain that $y_1 \geq 2^{2^{n-1}}$. 
 In this case, the optimal solution of the problem has low bit size but the vector $y$ has high bit size.
\end{ex}

\begin{ex}\label{ex:SDPnotfeasible}
Let $C_5$ be the cycle graph on $5$ vertices and let $A_{C_5}$ be its adjacency matrix. Let $I_5,J_5\in \mathcal S^5$ be the identity and the matrix of all ones, respectively. We define the following $6\times 6$ matrix
    $$M_1= \begin{pmatrix}
        2(A_{C_5}+I)-J  & 0\\
        0 & 0
       \end{pmatrix}.$$
This matrix is copositive (see formulation~\eqref{alpha}). Consider the copositive program
\begin{align}\label{ex-2}
    p^* = \min\Big\{ y: M_y=\left(\begin{array}{c|c} 
    M_1 & 0\\
    \hline 
    0 &  y 
    \end{array}\right) \in \COP_7\Big\}.
\end{align}
It is easy to observe that the optimal value of program (\ref{ex-2}) is 0. It was shown in~\cite{LV2023-exactness} that $M_1\notin \MK_6^{\smash{(r)}}$ for any $r\in \mathbb{N}$. It follows $M_y\notin \MK_7^{\smash{(r)}}$ for any $y\geq 0$ and $r\in\mathbb N$. This shows that the relaxations \eqref{cop-program-approx1} of program \eqref{ex-2} using the cones~$\MK_7^{\smash{(r)}}$ are not feasible. The same reasoning holds for the relaxations \eqref{cop-program-approx3} using the cones~$\MQ_7^{\smash{(r)}}$.
\end{ex}

\begin{ex}\label{ex:PBOSSDPneeded}
Let $A_{C_5}$, $I_5$ and $J_5$ in $\mathcal S^5$ be defined as in \Cref{ex:SDPnotfeasible}.
Given a vector $y\in\mathbb R^n$, and scalars $z,w \in \mathbb{R}$ we define the following sparse block matrix
    $$\tilde{M}_{wyz}=\begin{pmatrix}
              y_1 & -\frac{2}{3}y_2\\
              -\frac{2}{3}y_2 & 1
          \end{pmatrix} \oplus \dots \oplus \begin{pmatrix}
              y_{n-1} & -\frac{2}{3}y_{n}\\
              -\frac{2}{3}y_{n} & 1
          \end{pmatrix} \oplus \begin{pmatrix}
              y_{n} & -\frac{40}{3}(z-w)\\
              -\frac{40}{3}(z-w) & 1 
              \end{pmatrix}.
    $$
We consider the following problem 
\begin{equation}\label{ex-cop}
    p=\min \Big \{z +w : M_{wyz}=\left({\footnotesize{\begin{array}{c|c|c} 
            z(A_{C_5}+I_5)-J_5 & 0& 0\\
            \hline 
           0  & \Tilde M_{wyz} & 0\\
             \hline
             0 & 0 & w-2
            \end{array}}}\right) \in \COP_{2n+6} \Big \}.
\end{equation}
Observe that the top left block is copositive if and only if $z\geq 2$ (recall formulation~(\ref{alpha})). Also, for any feasible point, we have $w\geq 2$. It is easy to observe that the matrix $M_{wyz}$ obtained by taking $z=2$, $w=2$, and $y_1=y_2=\dots=y_{n-1}=0$, is copositive. Thus, this is an optimal solution of problem (\ref{ex-cop}). Now, we show that this problem satisfies property \eqref{intSPN}. Indeed, after picking  $z=6+(1/20)$, $w=6$, $y_1=y_2= \dots =y_n=1$, each block of the matrix $M_{wyz}$ takes the form $P+N$ for some matrices $P\succeq 0$, $N\geq 0$ of rational low bit size entries with $\lambda_{\min}(P)\ge 1/3$.

Now, we consider relaxation (\ref{cop-program-approx1}) of problem (\ref{ex-cop}) (with $r=0$).  
\begin{equation}\label{ex-0}
    p_{\MK}^{\smash{(0)}}=\min \Big \{z +w : M_{wyz}={\footnotesize{\left(\begin{array}{c|c|c} 
            z(A_{C_5}+I_5)-J_5 & 0& 0\\
            \hline 
           0  & \Tilde M_{wyz} & 0\\
             \hline
             0 & 0 & w-2
            \end{array}\right) }}\in \MK_{2n+6}^{\smash{(0)}} \Big \}.
\end{equation}
For every feasible solution we have that the top left block lies in $\MK_5^{\smash{(0)}}$. This implies that $z\geq \sqrt{5}$ \cite{dKP02}. Also, for any feasible solution it holds that $w\geq 2$.  We observe that if we pick $z=\sqrt{5}$, $w=2$, $y_n=4$ and $y_{i}\geq y_{i+1}^2$ for $i=1, \dots, n-1$, then all blocks of $\tilde{M}_{wyz}$ are positive semidefinite matrices, and thus this is an optimal solution for problem (\ref{ex-0}). Now, we show that for every optimal solution we have that $y_1\geq 2^{2^n}$. We set $z=\sqrt{5}$ and $w=2$. Then, by looking at the last block of $\tilde{M}_{wyz}$ we obtain that $y_n\geq 9$. Also, by looking at the $i$-th block of $\tilde{M}_{wyz}$, we get $y_i\geq (1/3)y_{i+1}^2$. Indeed,
    $$ \begin{matrix}
         (1 & y_{i+1}) \\
          & 
     \end{matrix}\left(\begin{matrix}
         y_i & -\tfrac{2}{3}y_{i+1} \\
         -\tfrac{2}{3}y_{i+1} & 1
     \end{matrix}\right)\left(\begin{matrix}      1\\
         y_{i+1}
     \end{matrix}\right) = y_i-\frac{4}{3}y_{i+1}^2 + y_{i+1}^2\geq 0.$$
This implies that $y_1\geq 3^{2^{n}-1}$. 

We have shown that every optimal solution of the relaxation (\ref{ex-0}) has doubly exponential magnitude. Thus, it cannot be computed (or approximated) in polynomial time.  However, the copositive program (\ref{ex-cop}) satisfies properties (\ref{PBOS}) and (\ref{intSPN}). Then, the strategy of Section \ref{new-bounds} can be applied. Namely, the upper bounds $p_\MQ^{\smash{'(r)}}$ and $p_\MK^{\smash{'(r)}}$  can be computed in polynomial time (up to fixed precision) (Theorem~\ref{thm-new-bound}).
\end{ex}

\appendix

\section{Reciprocal of a copositive program} \label{appendix:cop-inv}

Consider the standard quadratic problem 
\begin{align}\tag{SQP}
    p_{\min} =\min \Big\{x^TMx: x\in \Delta_n\Big\},
\end{align}
and the associated copositive program 
\begin{align}\tag{CP-INV}
    p^*=\min_{\lambda\in \mathbb{R}}\Big\{\lambda \colon \lambda M-J\in \COP_n\Big\}.
\end{align}

\begin{lemma}
    If $p_{\min}>0$, then $p^*=1/p_{\min}$.
\end{lemma}
\begin{proof}
     We first show that for every feasible point $\lambda\in \mathbb{R}$ of program \eqref{COP-INVERSE1}, we have $\lambda>0$. Indeed, since $p_{\min}>0$, we have that $x^TMx>0$ for all $x\in \Delta_n$. Thus, for any feasible $\lambda$, we have $x^T(\lambda M-J)x=\lambda x^TMx - 1 \geq 0$ for all $x\in \Delta_n$, implying that $\lambda > 0$. Now, for $\lambda >0$, we have:  
\begin{align*}
    \lambda M-J \in \COP_n & \Longleftrightarrow x^T(\lambda M -J)x \geq 0 \quad \forall x\in \Delta_n \\
    & \Longleftrightarrow \lambda x^TMx \geq 1 \quad \forall x\in \Delta_n \\
    & \Longleftrightarrow  x^TMx \geq \textstyle{\frac{1}{\lambda}} \quad \forall x\in \Delta_n \\ 
    & \Longleftrightarrow p_{\min} \geq \textstyle{\frac{1}{\lambda}}
 \end{align*}
On the one hand, this shows that $p_{\min} \geq 1/p^{*} $. On the other hand, it shows that $\lambda=1/p_{\min}$ is feasible for \eqref{COP-INVERSE1} and thus $p^*\leq 1/p_{\min}$. We conclude that $p^*=1/p_{\min}$.
\end{proof}

\section{Proof of equation \Cref{eq:rel q-p}} \label{appendix:rel q-p}

In the following, we prove equation \eqref{eq:rel q-p} for the relaxations \eqref{COP-SQ-K1} and \eqref{eq: sdp-knr1}. The same relation holds for the relaxations \eqref{COP-SQ-Q1} and \eqref{eq: sdp-qnr1} obtained using the cones $\MQ_n^{\smash{(r)}}$, and the proof is analogous to the one presented below.

Consider the equivalent formulation of problem \eqref{COP-SQ-K1} given by
\begin{equation} \label{COP-SQ-Kmax}
    p_\MK^{\smash{(r)}}=\max_{\lambda \in \mathbb R} \Big \{ \lambda \colon M - \lambda J\in \MK_n^{\smash{(r)}}\Big \}
\end{equation}
and let $\lambda>0$ be feasible for problem \eqref{eq: sdp-knr1}. It follows that $q_\MK^{\smash{(r)}}\leq \lambda$ for all $r\in \mathbb N$ and that 
$\lambda M -J \in \MK_n^{\smash{(r)}} \Longleftrightarrow M -(1/\lambda)J \in \MK_n^{\smash{(r)}} \Longleftrightarrow 1/\lambda \text{ feasible for problem \eqref{COP-SQ-Kmax}}$, 
which implies $p_{\MK}^{\smash{(r)}} \geq (1/\lambda)$. 
Observe that, taking $\lambda=q_\MK^{\smash{(r)}}$, we get 
\begin{equation}\label{eq:pkr-qkr01}
    p_{\MK}^{\smash{(r)}}\geq \textstyle{\frac{1}{q_\MK^{(r)}}} \quad \Rightarrow \quad q_\MK^{\smash{(r)}}\geq \textstyle{\frac{1}{p_{\MK}^{(r)}}}.
\end{equation}
Let now $\lambda^*> 0$ be feasible for problem \eqref{COP-SQ-Kmax}. Then, $p_{\MK}^{\smash{(r)}} \geq \lambda^*$ for all $r\in \mathbb N$, and
$M -\lambda^* J \in \MK_n^{\smash{(r)}} \Longleftrightarrow (1/\lambda^*)M -J \in \MK_n^{\smash{(r)}} \Longleftrightarrow 1/\lambda^* \text{ feasible for problem \eqref{eq: sdp-knr1}}$,
which implies $q_\MK^{\smash{(r)}} \leq 1/\lambda^*$. 
For $\lambda^*=p_{\MK}^{\smash{(r)}}$, we then get
\begin{equation}\label{eq:pkr-qkr02}
    q_\MK^{\smash{(r)}}\leq \textstyle{\frac{1}{p_{\MK}^{(r)}}}.
\end{equation}
Finally, putting inequalities \eqref{eq:pkr-qkr01} and \eqref{eq:pkr-qkr02} together, we derive 
$q_\MK^{\smash{(r)}}= 1/p_{\MK}^{(r)}.$

\section*{Acknowledgments}
We wish to thank Corbinian Schlosser for helpful discussions regarding \Cref{SEC:norms}, and in particular for pointing us to~\Cref{lemma-corbi}.

The first, third and fourth author are supported by the Swiss National Science Foundation project No.~200021\_207429 / 1 \textit{Ideal Membership Problems and the Bit Complexity of Sum of Squares Proofs}. The second author is supported by funding from the European Research Council (ERC) under the
European Union's Horizon 2020 research and innovation programme (grant agreement No.~815464).

\bibliographystyle{alpha}
\bibliography{Copositive_polynomial}

\end{document}